\documentclass[twoside]{article}

\usepackage{a4wide, amsmath,  amssymb, amscd, mathptmx, amsthm}
\usepackage{amsfonts}

\usepackage{color}

\usepackage{pstricks}
\usepackage[all]{xy}\CompileMatrices\SelectTips{cm}{12}

\theoremstyle{plain}
\newtheorem{Thm}{\sc Theorem}[section]
\newtheorem{theorem}[Thm]{\sc Theorem}
\newtheorem{corollary}[Thm]{\sc Corollary}
\newtheorem*{corollary*}{\sc Corollary}
\newtheorem{proposition}[Thm]{\sc Proposition}
\newtheorem*{proposition*}{\sc Proposition}
\newtheorem{lemma}[Thm]{\sc Lemma}

\newtheorem{question}[Thm]{\sc Questions}

\theoremstyle{remark}
\newtheorem{remark}[Thm]{Remark}
\newtheorem{example}[Thm]{Example}
\newtheorem*{example*}{Example}
\newtheorem*{remark*}{Remark}

\newcommand{\cA}{{\mathcal A}}
\newcommand{\cB}{{\mathcal B}}

\newcommand{\cE}{{\mathcal E}}
\newcommand{\cF}{{\mathcal F}}
\newcommand{\cG}{{\mathcal G}}

\newcommand{\cL}{{\mathcal L}}
\newcommand{\cM}{{\mathcal M}}

\newcommand{\cO}{{\mathcal O}}

\renewcommand{\AA}{{\mathbb A}}

\newcommand{\FF}{{\mathbb F}}
\newcommand{\GG}{{\mathbb G}}

\newcommand{\PP}{{\mathbb P}}
\newcommand{\QQ}{{\mathbb Q}}
\newcommand{\RR}{{\mathbb R}}
\newcommand{\VV}{{\mathbb V}}

\newcommand{\ev}{\mathop{\rm ev}}

\newcommand{\Bl}{{\mathop{\rm Bl\,}}}

\newcommand{\codim}{\mathop{\rm codim \, }}

\newcommand{\Der}{\mathop{\rm Der}}

\newcommand{\Exc}{\mathop{\rm Exc\, }}

\newcommand{\Flagf}{\underline{\mathop{\rm Flag}}}
\newcommand{\Flag}{{\mathop{\rm Flag}}}

\newcommand{\Gras}{\mathop{\rm Grass}}

\newcommand{\id}{{\mathop{\rm id}}}

\newcommand{\Fr}{{\mathop{\rm Fr}}}
\newcommand{\Spec}{\mathop{\rm Spec \, }}

\newcommand{\Proj}{\mathop{\rm Proj \, }}

\newcommand{\Sch}{{\mathop{{\rm Sch }}}}
\newcommand{\tors}{\mathop{{\rm Torsion}}}
\newcommand{\Hom}{{\mathop{{\rm Hom}}}}
\newcommand{\cHom}{{\mathop{{\cal H}om}}}

\newcommand{\GL}{\mathop{\rm GL}}
\newcommand{\Sym}{{\mathop{\rm Sym}}}
\newcommand{\Sing}{{\mathop{\rm Sing\,}}}

\begin{document}

\markboth{\rm A.\ Langer}{\rm  Geometry of Drinfeld half-spaces over a finite field}

\title{Birational geometry of compactifications of Drinfeld half-spaces over a finite field}
\author{Adrian Langer}

\date{\today}

\maketitle

{\sc Address:}\\
Institute of Mathematics, University of Warsaw,
ul.\ Banacha 2, 02-097 Warszawa, Poland\\
e-mail: {\tt alan@mimuw.edu.pl}

\medskip

\begin{abstract}
We study compactifications of Drinfeld half-spaces over a finite field. In particular, we construct a purely inseparable endomorphism of  Drinfeld's half-space $\Omega (V)$ over a finite field $k$ that does not extend to an endomorphism of the projective space $\PP (V)$. This should be compared with theorem of R\'emy, Thuillier and Werner  that every $k$-automorphism of $\Omega (V)$ extends to a $k$-automorphism of $\PP (V)$. Our construction uses an inseparable analogue of the Cremona transformation. We also study foliations on Drinfeld's half-spaces.
This leads to various examples of interesting varieties in positive characteristic. In particular, we show 
a new example of a non-liftable projective Calabi--Yau threefold in characteristic $2$ and we show examples of
rational surfaces with klt singularities, whose cotangent bundle contains an ample line bundle.
\end{abstract}


\let\thefootnote\relax\footnote{Author's work was partially supported by
Polish National Science Centre (NCN) contract numbers 2013/08/A/ST1/00804 and
2015/17/B/ST1/02634.}

\section{Introduction}

Let $k$ be a finite field and let $\Omega (k^{n+1})$ be Drinfeld's
half-space, i.e., the complement of all $k$-rational hyperplanes in
$\PP^n$. The main aim of the paper is to study various
compactifications of $\Omega (V)$ and their purely inseparable
quotients from the point of view of birational geometry. In this way we obtain 
a few interesting examples of varieties or maps defined over finite fields.

Our first result describes endomorphisms of Drinfeld's half-space.

\begin{theorem} \label{main-1}
\begin{enumerate}
\item Every separable dominant $k$-endomorphism $\psi$ of $\Omega
  (k^{n+1})$ is a $k$-automorphism and it extends to a
  $k$-automorphism of $\PP^n$.
\item There exists a radicial $k$-endo\-morphism $\psi$ of $\Omega
  (k^{n+1})$ that does not extend to a $k$-endomorphism of $\PP^n$.
\end{enumerate}
\end{theorem}

This should be compared with the main result of \cite{RTW} that says
that every $k$-automorphism of $\Omega (k^{n+1})$ extends to a
$k$-automorphism of $\PP^n$. The proof of this fact is done in two
steps: first the authors use Berkovich spaces to prove that every
$k$-automorphism of $\Omega (k^{n+1})$ extends to a $k$-automorphism
of the wonderful compactification $\tilde X$ of $\Omega (k^{n+1})$ and
then they use intersection theory to prove that one can descend this
automorphism to the projective space. We give a new proof of the first
step without Berkovich spaces following an idea that the author
learnt from H. Esnault and V. Srinivas. Namely, the result follows if
one proves that $K_{\tilde X}+D$ is ample for the complement $D$ of
$\Omega (k^{n+1})$ in $\tilde X$.  This fact goes back to Iitaka (at
least in the characteristic zero case; see \cite[Theorem 11.6]{Ii}).
Ampleness of $K_{\tilde X}+D$ is claimed by Mustafin in \cite[p.~227,
Lemma]{Mu}, but his proof shows only that it is strictly nef (he
checks only that the degree of $K_{\tilde X}+D$ on curves is
positive). We prove ampleness of this line bundle by showing that
$K_{\tilde X}+D=(|k|-1)L$ for some line bundle $L$ such that the complete
linear system $|L|$ is the composition of embedding of $\tilde X$ into
the full flag variety (as a Deligne--Lusztig variety), then embedding
the full flag variety into a product of projective spaces (via Pl\"ucker
embeddings of partial flags) and finally embedding this product by
the Segre embedding into some projective space.

To show the second part of the theorem we construct a new natural
compactification $X$ of $\Omega(k^{n+1})$ (which is non-normal if
$n\ge 3$) to which one can extend $\psi$, so that we have a
commutative diagram
$$ \xymatrix{
  X\ar[r]^{\varphi_X}\ar[d]^{\pi}&X\ar[d]^{\pi}\\
  \PP ^n\ar@{-->}[r]^{\psi}&\PP ^n\\
  \Omega (k^{n+1}) \ar[r]^{\psi}\ar@{^{(}->}[u] &\Omega (k^{n+1}) \ar@{^{(}->}[u] .  \\
}$$ in which $\varphi_X$ is a purely inseparable morphism such that
$\varphi_X^2$ is the composition of $(n-1)$ $k$-linear Frobenius
endomorphisms $\Fr_X$ of $X$.  This can be considered as a kind of an
``inseparable flop'' exchanging the role played by $k$-rational
codimension $i$ subspaces of $\PP^n$ and $k$-rational codimension
$(n+1-i)$ subspaces of $\PP^n$.  We also have a similar diagram with
$X$ replaced by the wonderful compactification $\tilde X$ of
$\Omega(k^{n+1})$ but then the corresponding endomorphism
$\varphi_{\tilde X}$ is not so easy to see directly.

Taking Stein's factorization of $\pi\circ \varphi_X$ we get a
birational morphism $f:X\to Y$ and a purely inseparable morphism
$\varphi_Y:Y\to \PP ^n$. The fact that $\varphi_X^2=\Fr_X^{\dim V-2}$
leads to a new purely inseparable morphism $\varphi: \PP ^n\to Y$. In
particular, we have the following more precise version of the second part of Theorem
\ref{main-1}:

\begin{theorem}  \label{main-3}
We have the following commutative diagram of $k$-schemes
$$ \xymatrix{
  Y\ar[d]^{\varphi_Y}&X\ar[l]_{f}\ar[r]^{\pi}\ar[d]^{\varphi_X}&\PP ^n\ar[d]^{\varphi}\\
  \PP ^n&X\ar[l]_{\pi}\ar[r]^{f}&Y\\
}$$ in which horizontal maps are birational, vertical maps are purely
inseparable and they satisfy the following relations
$$\varphi_X^2=\Fr _X^{n-1},\quad \varphi _Y\circ \varphi =\Fr_{\PP^n}^{n-1}, \quad \varphi\circ \varphi _{Y}=\Fr_Y^{n-1}.$$
The map $\psi=\pi\circ \varphi_X\circ \pi^{-1}: \PP^n\dashrightarrow
\PP ^n$ is purely inseparable and it satisfies
$\psi^2=\Fr_{\PP^n}^{n-1}$. Moreover, $\psi$ defines an endomorphism
of $\Omega(k^{n+1})$.
\end{theorem}

In the above theorem $\Fr$ stands for the $k$-linear Frobenius
endomorphism (``raising to $q$-th power'' for $k=\FF_q$). In fact, our
construction of schemes $X$, $Y$ and maps $\psi, \varphi_X$ etc. is
functorial in the $k$-vector space $V=k^{n+1}$ but the precise
statements become more complicated (see Theorem \ref{main-3'}) and we
get a morphism $\psi _V: \Omega (V)\to \Omega(V^*)$ rather than an
endomorphism of $ \Omega (V)$.

\medskip

We can look at $\psi$ as an analogue of the Cremona transformation,
which can be recovered by specializing to ``characteristic $1$''. We
illustrate this in dimension $n=2$. Let us recall that the Cremona
transformation is defined as the blow up of $3$ distinct points of
$\PP ^2$ and then blowing down the strict transforms of $3$ lines
between any two of the blown up points.  Similarly, $\psi$ is obtained
by first blowing up all $k$-rational points of $\PP ^2$, performing an
inseparable endomorphism $\varphi_X$ and then contracting images of
the strict transforms of all $k$-rational lines.  A slightly different
way to see the analogy is the following. The Cremona transformation
can be defined by choosing $3$ distinct points of $\PP ^2$ and taking
the map defined by products of equations of lines passing through $2$
of these points. To define $\psi$ in coordinates we need to choose $3$
distinct $k$-rational points of $\PP ^2$ and take the products of
equations of all $k$-rational lines passing through these points. For
example, for a standard choice of points $[1,0,0]$, $[0,1,0]$,
$[0,0,1]$ in $\PP^2$, the Cremona transformation is given
by $$[x_0,x_1,x_2]\to [x_1x_2, x_2x_0, x_0x_1],$$ whereas $\psi$ is
given by
$$[x_0,x_1,x_2]\to [w_0,w_1,w_2]=[x_1x_2^q-x_1^qx_2 , x_2x_0^q-x_0x_2^q, x_0 x_1^q-x_1x_0^q].$$
This definition makes it rather non-obvious that $\psi^2$ is given by
$[x_0,x_1,x_2]\to [x_0^q,x_1^q,x_2^q]$ or that $\psi\circ \pi$ factors
through a purely inseparable morphism $X\to X$.

\medskip

We construct $X$ as a complete intersection in the product of
two projective spaces. Then we can easily construct $\varphi_X$ as
restriction of a certain morphism on the product. Finally, we check
that it preserves the boundary divisor producing an endomorphism of
$\Omega (k^{n+1})$.

\medskip

A large part of the paper is devoted to the study of $1$-foliations related to Drinfeld's compactification.
More precisely, the logarithmic tangent sheaf of $\PP^n$ with poles along complement of $\Omega (k^{n+1})$
splits into a direct sum of line bundles. The direct sums of these line bundles define $1$-foliations
of arbitrary rank on $\PP ^n$ and singularities of these $1$-foliations can be resolved by successive 
blowing ups along $\FF_q$-rational linear spaces of fixed codimension. 
In case $k=\FF_{p}$ and rank $1$ we obtain the $1$-foliation considered in \cite{Hi} to construct a non-liftable Calabi-Yau threefold in characteristic $3$.  For $k=\FF_{2}$  we can use our rank $2$ $1$-foliation to
construct a new example of a non-liftable projective  Calabi-Yau threefold.

\begin{theorem}\label{new-CY}
  There exists a smooth, projective Calabi--Yau $3$-fold $\tilde Y/\bar
  \FF_2$ such that the following conditions are satisfied:
\begin{enumerate}
\item $b_2(\tilde Y)=51$ and $h^1(\cO_{\tilde  Y})=h^2(\cO_{\tilde Y})=0$,
\item  $\tilde  Y$ is unirational,
\item $h^0(T_{\tilde Y})=0$, 
\item $T_{\tilde  Y}$ is not semistable with respect to some ample polarizations,
\item $\tilde Y$ does not admit a formal lifting to characteristic zero.
\end{enumerate}
\end{theorem}

In the above theorem $b_i(X)$ denotes the $i$-th Betti number of $X$,
i.e., the dimension of $i$-th $l$-adic cohomology group of $X$ for
some $l\ne p$.

\medskip

Note that till now the only smooth projective Calabi--Yau $3$-folds that are not liftable to characteristic zero 
were constructed in characteristic $2$ and $3$ (see, e.g., \cite{Hi}, \cite{Sc}, \cite{CS}). There exist also 
non-projective examples in other characteristics (see \cite{CS}). It seems that constructions 
in characteristic $2$ are somewhat more complicated than those in characteristic $3$.

\medskip

Before stating other applications of our result, let us recall that in
any characteristic the cotangent bundle of a separably uniruled smooth
proper variety does not contain any big line bundles (see
\cite[Chapter V, Lemma 5.1]{Ko}). This was used by J. Koll\'ar (see
\cite[Chapter V, Theorem 5.14]{Ko}) and in various subsequent papers
to prove non-rationality of some complex Fano varieties. Moreover, it
is known that if $Y$ is a complex projective variety with only klt
singularities then $\hat \Omega_Y=(\Omega_Y)^{**}$ does not contain
any big $\QQ$-Cartier Weil divisors (see \cite[Theorem 7.2]{GKKP}).
Here we show that in positive characteristic $\hat \Omega_Y$ can
contain ample line bundles, even if $Y$ has only klt singularities and
it is separably uniruled. More precisely, we have the following
theorem:

\begin{theorem}\label{B-S-failure}
Let $k=\FF_{q}$ be a finite field of characteristic $p$. There exists a geometrically normal
projective surface $Y_0/k$ such that $Y=(Y_0)_{\bar k}$ has
the following properties:
\begin{enumerate}
\item $Y$ is rational,
\item $Y$ has only klt singulartities of type a cone over $(\PP^1, \cO_{\PP^1}(q))$,
\item $\rho (Y)=1$,
\item $\hat \Omega_Y$ contains an ample line bundle,
\item $T_Y$ is not generically semi-negative,
\item if $p=2$ then  $-K_Y$ is ample; moreover, if $q=2$ then  $T_Y$ is locally free,
\item if $p>2$ then  $K_Y$ is ample.
\end{enumerate}
\end{theorem}

This theorem is related to the following proposition:

\begin{proposition} \label{main-2}
  Let $g: \tilde S=\VV (\cO_{\PP^1}(-d))\to S=\Spec \bigoplus _{i\ge
    0} H^0(\PP^1, \cO(id))$ be the resolution of a cone over $(\PP^1,
  \cO_{\PP^1}(d))$ in characteristic $p$ dividing $d$. Then $S$ has
  only klt, $F$-regular singularities but $g_*\Omega_{\tilde S}$ is
  not reflexive.
\end{proposition}

This should be contrasted with the main result of \cite[Theorem
1.4]{GKKP}, which says that over complex numbers if $g: \tilde Z\to Z$
is a log resolution a quasi-projective variety $Z$ with at most klt
singularities then $g_*\Omega_{\tilde Z}$ is reflexive.

\medskip

\subsection*{Relation to other work.}

Let $V$ be an $(n+1)$-dimensional vector space over a finite field $k$.
After most of the paper was written, the author noticed a related
paper \cite{PS} of R. Pink and S. Schieder in which the authors
construct some purely inseparable maps $g_V:\Omega (V)\to \Omega
(V^*)$ and $f_V:\Omega (V^*)\to \Omega (V)$ (see \cite[Proposition
9.1]{PS}; note that most of the proof is skipped and ``lengthy
elementary calculations'' are left to the reader).

Note that in their case $f_Vg_V= \Fr _{\Omega (V)}^n$ and $g_Vf_V= \Fr
_{\Omega (V^*)}^n$, whereas we construct a map $\psi_V: \Omega(V)\to \Omega (V^*)$ 
such that $\psi_{V^*}\circ \psi _V=\Fr_{\Omega (V)}^{n-1}$ (see  Theorem \ref{main-3'}, which 
is a more precise version of Theorem \ref{main-3}). It is rather easy to see that $f_V=\psi_V$ and after 
asking about relation between the results, the author was informed by R. Pink that after 
some computations one can show that $g_{V^*}=f_{V}\circ \Fr_{:\Omega (V^*)}$. In particular,   
\cite[Proposition 9.1]{PS} agrees with Theorem \ref{main-3'}.
Our approach and description of the map $\psi _V$ is different and more geometric 
than that of $f_V$, showing its relation to the Cremona transformation.

S. Kurul and A. Werner proved in a recent preprint \cite{KW} that all
endomorphisms of Drinfeld's half-space extend to the wonderful
compactification. In fact, their method allows to deal also with more
general complements of hyperplane arrangements in $\PP ^n$.

Theorem \ref{B-S-failure} comes from a detailed study of the example
in \cite[Section 8]{La} and from the related paper \cite{CT} by
Cascini and Tanaka.  More precisely, in \cite{La} the author showed
that if $D$ on $X$ considered above is the sum of strict transforms of
$k$-rational lines then $\Omega_X(\log D)$ contains a big line bundle
with positive self-intersection. However, this was only an existence
result proven using non-trivial methods. Here we produce such a line
bundle bundle explicitly using some computations involving certain foliation on $X$. 
Then we push it down to surface $Y$ obtained by
contracting $D$ and studied in \cite{CT} (see \cite[Lemma 2.4]{CT}).
\medskip

\subsection*{Notation}

Let $k$ be a finite field with $q=p^e$ elements and let $\Fr _k: \bar
k\to \bar k$ be the Frobenius automorphism of $\bar k$ defined by
$\Fr_k(a)=a^q$.  Clearly, $\Fr_k$ is identity on $k\subset \bar k$.

Let $X$ be a scheme of characteristic $p$. The absolute Frobenius morphism
$F_X:X\to X$ is defined as identity on topological spaces and raising
to $p$-th power on structure sheaves. For a $k$-scheme $X$ we define a
$k$-scheme $X^{(i)}$ as the base change of $X/k$ via $k\to k$ defined
by $a\to a^{p^i}$. Then we get a $k$-linear morphism $F^i: X\to
X^{(i)}$ induced by $(F_X)^i:X\to X$, i.e., composition of $i$
absolute Frobenius morphisms $F_X$.  Since $\Fr _k=(F_k)^e$ is
identity on $k$, the $k$-schemes $X$ and $X^{(e)}$can be identified
and we have a $k$-linear endomorphism $\Fr_X=F^e: X\to X^{(e)}=X$.

\medskip

If $V$ is a finite dimension $k$-vector space then we set
$\PP (V)=\Proj (\Sym ^{\bullet} V)$.  More generally, if $E$ is a
locally free sheaf on a scheme $X$ then we set $\PP (E)=\Proj (\Sym
^{\bullet} E)$.

A scheme structure on a $k$-vector space $W$ is given by $W=\Spec \Sym
^{\bullet} W^*$.  By definition we have the standard projection map
$V^*-\{0\}\to \PP (V)$. This is a quotient by the canonical
$\GG_m$-action on $V^*$ (which is free on $V^*-\{0\}$).  Linear
coordinates in $V^*$ (i.e., affine coordinates vanishing at $0$ or a
choice of basis of $V$) give homogeneous coordinates in $\PP (V)$.  If
we choose homogeneous coordinates $x_0,...,x_n$ in $\PP (V)$ and take
an element $f\in k[x_0,...,x_n]$ then we denote by $D_{+}(f)$ the
principal open subset corresponding to $f$, i.e., $D_{+}(f)=\{\mathfrak p
\in \PP(V)=\Proj k[x_0,...,x_n] | \, f \not \in \mathfrak p \}$.

\section{Preliminaries}

\subsection{Moore determinant}

The \emph{Moore determinant} $\Delta_q(w_1,...,w_n)$  is defined as the determinant of the matrix
$$\left(
\begin{array}{ccc}
w_1&...&w_n\\
w_1^q&...&w_n^q\\
\vdots&& \vdots\\
w_1^{q^{n-1}}&...&w_n^{q^{n-1}}\\
\end{array}
\right)
$$
treated as an element of $\FF_q[w_1,...,w_n]$.

\begin{lemma} (see \cite[Corollary 1.3.7]{Go})\label{Moore}
Moore determinant is equal to the product of all linear forms over $\FF_q$ with the last non-zero coefficient equal to $1$, i.e., we have
$$\Delta_q(w_1,...,w_n)=\prod _{i=1}^n\prod _{a_{1}\in \FF_q}...\prod _{a_{i-1}\in \FF_q}(a_1w_1+...+a_{i-1}w_{i-1}+w_i)$$
in $\FF_q[w_1,...,w_n]$.
\end{lemma}

\begin{corollary}\label{partial-Moore}
We have
$$\frac{\partial \Delta _q}{\partial w_i} (w_1,...,w_n)=(-1)^{i+1}
(\Delta _q(w_1,...,\hat{ w_i},...,w_n))^q.$$
\end{corollary}

\begin{proof}
Using expansion of the defining matrix with respect to the $i$-th column, we get
$$\frac{\partial \Delta _q}{\partial w_i} (w_1,...,w_n)=
(-1)^{i+1} \Delta _q(w_1^q,..., \hat {w_i^q},...,w_n^q).$$
Since $$a_1w_1^q+...+a_{i-1}w_{i-1}^q+w_i^q=(a_1w_1+...+a_{i-1}w_{i-1}+w_i)^q$$
for $a_1,...,a_{i-1}\in \FF_q$, the required equality follows from Lemma \ref{Moore}.
\end{proof}

\subsection{Graphs of rational maps and blow-ups}

Let $X$ be an integral scheme defined over some field $k$ and let $\cL$
be an invertible sheaf on $X$. Let us consider a finite dimensional $k$-subspace
$W\subset H^0 (X, \cL)$ and let
$$u: W\otimes _k \cO_X\subset H^0 (X, \cL)\otimes _k \cO_X\to \cL$$ be the
restriction of the evaluation map. The image of the induced map
$W\otimes _k\cL^{-1}\to \cO_X$ is an ideal sheaf of some subscheme
$Y\subset X$. The following lemma is certainly well known but we prove it for lack of an appropriate reference.

\begin{lemma}\label{blow-up}
The blow up of $X$ along $Y$ is the schematic closure in $X\times _k\PP (W)$ of the graph of the map
$U=X-Y\to \PP (W)$ defined by the linear system corresponding to $W$.
\end{lemma}

\begin{proof}
  Let $\cA=\bigoplus _{d\ge 0}I_Y^d$ and let $\cA*\cL$ be a sheaf of
  graded $\cO_X$-algebras defined by $(\cA*\cL)_d=\cA_d\otimes _{\cO_X}\cL
  ^{\otimes d}$ for $d\ge 0$.
The surjective map $W\otimes _k \cO_X\to I_Y\otimes _{\cO_X} \cL$ defines a surjective map of sheaves of graded $\cO_X$-algebras
$$\Sym^{\bullet} (W\otimes _k\cO_X)\to \cA *\cL , $$
that preserves degrees. It gives rise to a closed immersion
$$Z:=\Proj \cA \simeq \Proj \cA*\cL \hookrightarrow \Proj \Sym^{\bullet} (W\otimes _K\cO_X)=\PP (W)_X
$$
of the blow up $Z$ of $X$ along $Y$ into $\PP (W)_X$. Since $X$ is integral, $Z$ is also integral.
Therefore $Z$ is the scheme-theoretic closure in $X\times_X\PP (W)_X=\PP (W)_X$ of the graph of the morphism
$U=X-Y\to \PP (W)_X$
defined by surjection $W\otimes _k \cO_U\to \cL|_U$. Since $\PP (W)_X=X\times _k\PP (W)$,
this is equivalent to taking the scheme-theoretic closure of the graph of the morphism $U\to \PP (W)$ defined by the linear system corresponding to $W$.
\end{proof}

\medskip

\begin{lemma}\label{blow-up-sum}
Let $X$ be a scheme and let  $Y_1, Y_2\subset X$ be subschemes. Let $f_i: Z_i\to X$ be the blow-ups of $X$ along $Y_i$ for $i=1,2$ and let $W_i$ be  the blow up of $Z_i$ along $f_i^{-1}(Y_{j})$, where $\{i,j\} =\{1,2\}$.
Then $W_1$ is isomorphic to $W_2$ as an $X$-scheme and both of them dominate the blow up of $X$ along the subscheme $Y_1\cup Y_2\subset X$.
\end{lemma}

\begin{proof}
The first part of the lemma is standard (see, e.g., \cite[Lemma 3.2]{Li}). To prove the second one note that  both $(g_1\circ f_1)^{-1}(Y_1)\subset W_1$ and  $(g_1\circ f_1)^{-1}(Y_2)\subset W_1$ are Cartier divisors. Since
the ideal sheaf of $Y_1\cup Y_2$ is the intersection of ideal sheaves of $Y_1$ and $Y_2$ in $\cO_X$, also
 $(g_1\circ f_1)^{-1}(Y_1\cup Y_2)\subset W_1$ is a Cartier divisor. Hence the required  assertion follows from the universal property of the blow-up.
\end{proof}

\subsection{Wonderful compactifications}\label{wonderful-sec}

Let $X$ be a smooth algebraic variety defined over some perfect field $k$. An \emph{arrangement} of subvarieties $\cL$ is a finite collection of smooth subvarieties such that all nonempty scheme-theoretic intersections of subvarieties in $\cL$ are in $\cL$. Any such arrangement is a poset with order defined by inclusion.
 Let us set $X^o:=X-\bigcup _{Y\in \cL}Y$. A subset $\cB\subset \cL$ is called a \emph{building set} if for all $Y\in \cL-\cB$ the minimal elements in $\{Z\in \cB: Y\subset Z\}$ intersect transversally and their intersection is equal to $Y$.

For any arrangement $\cL$ one can define the \emph{wonderful compactification} of $X^o$ as the scheme-theoretic closure of the image of the natural locally closed embedding
$$X^o\hookrightarrow \prod _{Y\in \cL}\Bl _YX.$$
The following theorem is a summary of various results starting with \cite{DP}, \cite{Hu} and finishing with \cite[Theorem 1.2 and Theorem 1.3]{Li}. We state it only in a special case when a building set is equal to the arrangement. This is the only case that is interesting from the point of view of compactification of Drinfeld's half-space.

\begin{theorem}\label{wonderful-comp}
Wonderful compactification $\tilde X$ of $X^o$ is a smooth algebraic $k$-variety isomorphic to the blow up of $X$ along the product of ideal sheaves of $Y\in \cL$. For each $Y\in \cL$ there is a smooth divisor $D_Y\subset \tilde X$ such that $\tilde X-X^0=\bigcup _{Y\in \cL} D_Y $ is a simple normal crossings divisor and an intersection of divisors $D_{Y_1}\cap D_{Y_2}\cap ...\cap D_{Y_m}$ is nonempty if and only if $D_{Y_1}$, $D_{Y_2}$,..., $D_{Y_m}$
form a chain in the poset $\cL$.
Moreover, if we order elements of $\cL$ in such a way that for any $i$ the first $i$ elements of the order $D_1,....,D_N$ form a building set then
$$\tilde X=\Bl_{\tilde D_N}...\Bl_{\tilde D_2}\Bl_{D_1}X,$$
where $\tilde D_{i}$ are defined inductively as $f_i^{-1}(D_i)$ for $f_i: \Bl_{\tilde D_{i-1}}...\Bl_{D_1}X\to X$.
\end{theorem}

\subsection{Foliations}

Let $X$ be a normal variety defined over some field $k$. A
\emph{foliation} on $X$ is a saturated $\cO_X$-submodule $\cF\subset
T_{X/k} =\Der _k(\cO_X, \cO_X)$ preserved by the Lie bracket. If
$\cF\subset T_{X/k}$ is a foliation then by definition $T_{X/k}/\cF$
is torsion free and hence $\cF$ is reflexive. In particular, if $X$ is
a surface then $\cF$ is locally free. If $X$ is smooth and
$T_{X/k}/\cF$ is locally free then we say that $\cF$ is a \emph{smooth
  foliation}.

Let $f: X_2\to X_1$ be a birational morphism between normal
$k$-varieties. Then a foliation $\cF_1\subset T_{X_{1}/k}$ induces on
$X_2$ a foliation $\cF_2\subset T_{X_{2}/k}$ (of the same rank as
$\cF_1$) by taking the kernel of
$$T_{X_2/k}\to f^*T_{X_1/k}\to f^*(T_{X_1/k}/\cF_1)/\tors .$$
By assumption the first map is injective, so if $\cF_1$ is locally
free then $\cF_2$ is the saturation of $f^*\cF_1\cap T_{X_2/k}$ in
$T_{X_2/k}$.  Note that condition $[\cF_1, \cF_1]\subset \cF_1$
implies that $[\cF_2, \cF_2]\subset \cF_2$. Indeed, the map $\bigwedge
^2\cF_2\to T_{X_2/k}/\cF_2$ is $\cO_{X_2}$-linear and zero at the
generic point of $X_2$, so it is the zero map.

The above construction shows that we can try to resolve singularities
of a foliation $\cF_1\subset T_{X_{1}/k}$ and try to construct a
proper birational map $f: X_2\to X_1$ from a smooth variety $X_2$ such
that the induced foliation is smooth. In general, this is not possible
even for foliations on smooth complex surfaces (see
\cite{Se}). However, all the foliations considered in this paper have
nice resolutions by smooth foliations.

\subsection{$1$-foliations in positive characteristic}

Let $X$ be a normal variety defined over some field $k$ of positive
characteristic $p$. If $x$ is a $k$-derivation of the sheaf of rings
$\cO_X$ then the differential operator $x^p$ acts on $\cO_X$ as a
derivation. The corresponding derivation is denoted by $x^{[p]}$.
Thus we have a well defined $p$-th power map  $\cdot^{[p]}:T_{X/k}\to T_{X/k}$.

\medskip

We say that a foliation $\cF\subset T_{X/k}$ is a
\emph{$1$-foliation}, if $\cF ^{[p]}\subset \cF$.

If $f: X_2\to X_1$ is a birational morphism of smooth $k$-varieties
then a $1$-foliation $\cF_1\subset T_{X_{1}/k}$ induces a
$1$-foliation $\cF_2\subset T_{X_{2}/k}$. To check this one needs to
remark that the map $F_{X_2}^*\cF_2\to T_{X_2/k}/\cF_1$ is
$\cO_{X_2}$-linear and since $\cF ^{[p]}_1\subset \cF_1$, this map is
zero at the generic point of $X_2$. Therefore it is the zero map and
hence $\cF_2 ^{[p]}\subset \cF_2$.

\medskip

In positive characteristic, there always exist quotients by
$1$-foliations. More precisely, let $\cF\subset T_{X/k}$ be a
$1$-foliation. Then we can define a sheaf $\cO_Y=\cO_X^{\cF}$ of
abelian groups by setting
$$\cO_Y (U)=\{x\in \cO_X(U): \forall \,  V\subset U \, \hbox{ open, } \forall \, s\in \cF (V) \quad  s(x|_V)=0\}$$
for an open subset $U\subset X$. One can check that $\cO_Y$ is a sheaf
of $k$-algebras so it defines a scheme $Y$, which has $X$ as the
underlying topological space and $\cO_Y$ as a structure sheaf.  The
inclusion $\cO_Y\subset \cO_X$ defines the map $f: X\to Y$ called a
\emph{quotient by $1$-foliation} $\cF$. Usually, we denote $Y$ by
$X/\cF$. Since $\cO_X^p\subset \cO_Y$, the $p$th power map
$F_X^{\#}:\cO_X\to \cO_X$ factors through $\cO_Y\subset \cO_X$, so
there exists a map $g: Y\to X$ such that $gf=F_X$. Clearly, also
$F_Y^{\#}:\cO_Y\to \cO_Y$ factors through $\cO_Y\subset \cO_X$, so
$fg=F_Y$.

Let us define  the map of sheaves of abelian groups
$$\varphi : \cO_X\to \cF^*=\cHom _{\cO_X} (\cF, \cO_X)$$
by $(\varphi (x)) (s)=s(x)$ for $x\in \cO_X$ and $s\in \cF \subset \Der _k(\cO_X, \cO_X)$.
The map $\varphi$ can be also described as a composition of the  universal derivation $d:\cO_X\to \Omega_{X/k}$
with an $\cO_X$-linear map $\Omega_{X/k}\to \cF ^*$ dual to the inclusion
$\cF\subset T_{X/k}$. Note that $f_*(\varphi):  f_*\cO_X\to f_*(\cF ^*)$
is $\cO_Y$-linear. Indeed, if $y\in \cO_Y$, $x\in \cO_X$ and $s\in \cF$ then
$$(\varphi (yx)) (s)=s(yx)=ys(x)+xs(y)=ys(x).$$
Therefore we have the following lemma:

\begin{lemma}\label{foliation-sequence}
If $f: X\to Y$ is a quotient by $1$-foliation $\cF$ then
$$0\to \cO_Y\to f_*\cO_X\to f_*(\cF ^*)$$
is an exact sequence of $\cO_Y$-modules.
\end{lemma}

\medskip

\begin{lemma} \label{Q-factoriality} Let $X$ be a $\QQ$-factorial
  normal variety and let $\cF\subset T_{X/k}$ be a $1$-foliation.
  Then the quotient $Y=X/\cF$ is also a $\QQ$-factorial normal
  variety. Moreover, if $X$ is smooth and $D$ is a Weil divisor on $Y$
  then $pD$ is Cartier.
\end{lemma}

\begin{proof}
  Normality of $Y$ is well-known.  Let $D$ be a Weil divisor on
  $Y$. Let us recall that we have finite morphisms $f:X\to Y$ and
  $g:Y\to X$ such that $fg=F_Y$ and $gf=F_X$. Since $f^*D$ is
  $\QQ$-Cartier and $pD=F_Y^*D=g^*(f^*D)$, the divisor $D$ is also
  $\QQ$-Cartier. If $X$ is smooth then $f^*D$ is Cartier, which
  implies that $pD$ is Cartier.
\end{proof}

\subsection{Hirokado's rational vector field}\label{Hirokado}

Let $V$ be a finite dimension $k$-vector space.
Let us consider the Euler exact sequence
$$0\to \cO_{\PP (V)}(-1)\mathop{\longrightarrow}^{s} V^*\otimes \cO_{\PP (V)}\to T_{\PP (V)}(-1)\to 0,$$
where $s$ is dual to the evaluation map $V\otimes \cO_{\PP (V)}\to
\cO_{\PP (V)}(1)$.  Pulling back this sequence by the Frobenius
endomorphism $\Fr: \PP(V)\to \PP (V)$ we have
$$0\to \cO_{\PP (V)}(-q)\mathop{\longrightarrow}^{\Fr^*s} V^*\otimes \cO_{\PP (V)}\to ({\Fr}^*T_{\PP (V)})(-q)\to 0.$$
Let us denote  by $\theta$ the composition
$$ \cO_{\PP (V)} (-q)\mathop{\longrightarrow}^{\Fr^*s}   V^*\otimes \cO_{\PP (V)}\to T_{\PP (V)}(-1)
$$
twisted by the identity map on $\cO_{\PP (V)}(1)$. This defines a
rational vector field on $\PP (V)$ with poles of order $(q-1)$.  We
denote it by $\delta$.  This gives a global, coordinate free
construction of an analogue of the vector field considered by Hirokado
in \cite[Proposition 2.1]{Hi} in case of $\PP ^3$.  The following
lemma is a generalization of the first part of \cite[Proposition
2.1]{Hi}

\begin{lemma}\label{lemma-Hirokado}
  $\delta$ defines a $p$-closed rational vector field on $\PP(V)$
  singular along $k$-rational points of $\PP (V)$.  More precisely, we
  have $\delta^{[p]}=\delta$.
\end{lemma}

The proof of this lemma in \cite{Hi} was omitted but it can be found
in the preprint version of the paper.

\section{Determinantal schemes related to Moore's determinant}\label{determinantal}

Let us recall that a codimension $c$ subscheme of $\PP ^n$ is
called \emph{standard determinantal} if its homogeneous saturated
ideal is generated by the $m\times m$ minors of some $m\times (m+c-1)$
homogeneous matrix.

\medskip

Let $V$ be an $(n+1)$-dimensional vector space over a finite field
$k=\FF_q$, where $q=p^e$ for some prime number $p$ and a positive integer $e$.
Let us fix some homogeneous coordinates $x_0,...,x_n$ in $\PP (V)$.
In the following for a subset $I\subset \{0,...,n\}$ we denote by $|I|$ the number of elements in $I$ and we set
$\hat I= \{0,...,n\} -I$. If $|I|=m$ then we order elements of $I$ as  $0\le i_1<...<i_{m}\le n$ and we write
$x_{I}=({x}_{i_1}, ...,{x}_{i_{m}})$.

\begin{proposition} \label{codim-2} Let us fix $c\ge 1$ and let
  $Z_c\subset \PP (V)$ be the reduced induced subscheme structure on
  the union of all $k$-linear subspaces of codimension $c$ in $\PP
  (V)$.  Then after fixing some homogeneous coordinates $x_0,...,x_n$
  in $\PP (V)$ the saturated homogeneous ideal $I(Z_c)\lhd
  k[x_0,...,x_n]$ is generated by $\Delta_{q}(x_{\hat I})$ for all
  multiindices $I$ such that $|I|=c-1$.  Moreover, $Z_c\subset \PP
  (V)$ is a standard determinantal scheme. In particular, it is
  arithmetically Cohen--Macaulay.
\end{proposition}

\begin{proof}
  Let us fix some coordinates $x_0,...,x_n$ in $V^*$ that we will also
  treat as homogeneous coordinates $x_0,...,x_n$ in $\PP (V)$. So we
  have $V^*=\Spec S$ and $\PP (V)=\Proj S$, where $S=k[x_0,...,x_n]$.
  Let us consider the affine subscheme $Z_c'\subset V^*$ defined by
  the homogeneous ideal $J\lhd S$ generated by the maximal minors of
  the $(n+2-c)\times (n+1)$ homogeneous matrix
$$
\left(
\begin{array}{ccc}
x_0&...&x_n\\
x_0^{q}&...&x_n^{q}\\
\vdots&& \vdots\\
x_0^{q^{n+1-c}}&...&x_n^{q^{n+1-c}}\\
\end{array}
\right) .
$$
These minors are exactly of the form $\Delta_{q}(x_{\hat I})$ for multiindices $I$  with $|I|=c-1$.
We claim that $Z'_c$ is a codimension $c$ subscheme, whose underlying set of points
is the union of all $k$-linear vector subspaces of codimension $c$ in $V^*$.

By definition $Z_c'$ is the scheme-theoretic intersection of
$k$-subschemes $Z_{I}'\subset V^*$, defined by the ideal generated by
$f_I:=\Delta_{q}(x_{\hat I})$ for all multiindices $I$  with $|I|=c-1$.
Let us set $P_i=(0,...,\mathop{1}^{i},..., 0)\in V^*$.
By Lemma \ref{Moore} $Z_{I}'$ is the reduced scheme structure on
the set of all $k$-linear hyperplanes passing through $0$ and
$\{P_{i} \}_{i\in I}$.

Let $L\subset V^*$ be a $k$-vector subspace of codimension
$c$. The  subspace spanned by $L$ and $0, P_{i_1},..., P_{i_{c-1}}$ is $k$-linear
of dimension at most $n$. Therefore it is contained in some $k$-linear hyperplane and hence
 $L\subset \cap Z_{I}'=Z_c'$. On the other hand,
it is clear that if a point $x\in V^*$ does not lie on any $k$-vector
subspace of $V^*$ of codimension $c$, then $x \not \in \cap
Z_{I}'$. Therefore $Z'_c\subset V^*$ has codimension $c$.

Thus we proved that $J$ is a standard determinantal ideal. Therefore
by the Eagon--Hochster theorem (or just Eagon's theorem in this case)
the ring $S/J$ is Cohen--Macaulay.  In particular, $Z'_c$ has no
embedded components. We need to check that $Z'_c$ is reduced.  To do
so it is sufficient to check that it is reduced at the generic point
of each of its irreducible components. The irreducible components of
$Z_c'$ are codimension $c$ $k$-vector subspaces in $V^*$. Let us fix
such a subspace $L$. We can choose coordinates so that the ideal of
this subspace is equal to $(x_{n-c+1},..., x_n)\lhd S$. We need to
check that the Jacobian matrix
$$J_{\{f_I\}}=\left( \frac{\partial f_I}{\partial x_j}\right) _{I, j}$$
has rank $c$ at the generic point of $L$.
Let us note that for $n+1-c\le l\le n$ we have
$$\frac{\partial f_{n+1-c,...,\hat{j}, ...,n}}{\partial x_l}=\frac{\partial \Delta_{q}
(x_0,...,{x}_{n-c},x_j)}{\partial x_l}=\left\{
\begin{array}{cl}
(-1)^{n+1-c}(\Delta_q (x_0,...,x_{n-c}))^q&\hbox{for $l= j$,}\\
0&\hbox{for $l\ne j.$}\\
\end{array}
\right.$$
Therefore $J_{\{f_I\}}$ contains a $c\times c$ diagonal matrix
$$\left(\frac{\partial f_{n+1-c,...,\hat{j}, ...,n}}{\partial x_l}\right) _{n+1-c\le j,l\le n},$$
whose determinant is equal to $(-1)^{c(n+1-c)}(\Delta
(x_0,...,x_{n-c}))^{cq}$. Since this determinant is non-zero at the
generic point of $L$, $Z'_c$ is reduced at the generic point of $L$
and thus also reduced.  It follows that the graded ideal $J\lhd S$ is
radical. It is easy to see that graded radical ideals in $S$ different
to the ideal $(x_0,...,x_n)$ are saturated, so $J\lhd S$ is saturated.

It follows that $Z'_c$ is the affine cone over $Z_c$ and $J\lhd S$ is
also the homogeneous saturated ideal of $Z_c$. In particular, $Z_c$ is
a standard determinantal scheme and it is arithmetically
Cohen--Macaulay.
\end{proof}

\medskip

Note that we have a flag of determinantal schemes
$Z_{n+1}:=\emptyset\subset Z_n\subset...\subset Z_1\subset Z_0:=\PP
(V)$, showing that all $Z_c$ are so called good determinantal
subschemes. Moreover, this flag defines a partition $\PP
(V)=\bigcup_{i=0}^n(Z_i-Z_{i+1})$, where $\bar k$-rational points of
$Z_{i}-Z_{i+1}$ are exactly those $x\in (\PP(V))(\bar k)$ for which
$x, \Fr (x), \Fr ^2(x),...$ span a linear subspace of codimension $i$
in $\PP (V)$. In particular, this partition is invariant under the
Frobenius endomorphism on $\PP (V).$

\medskip

Let us consider evaluation map $\ev: V\otimes \cO_{\PP (V)}\to \cO_{\PP (V)}(1)$ and its pull backs
$(\Fr ^i)^* \ev: V\otimes \cO_{\PP (V)}\to \cO_{\PP (V)}(q^i)$ by the
Frobenius endomorphisms $\Fr ^i: \PP(V)\to \PP (V)$.
Let us fix  $c$ (``codimension'') and consider the map
$$g: V\otimes \cO_{\PP (V)}\to \cF:=\bigoplus _{i=0}^{n+1-c}\cO_{\PP (V)}(q^i)$$
given by the direct sum of the above maps.

\medskip

Proposition \ref{codim-2} allows us to write down a free resolution of the homogeneous
coordinate ring of $Z_c$. We write down just the corresponding sheafified version
of the Eagon--Northcott complex (see \cite[Chapter 2.C]{BV}).

\begin{corollary}\label{Eagon-Northcott}
For $1\le c\le n+1$  we have a natural acyclic complex
\begin{eqnarray*}
0\to {\bigwedge} ^{n+1} V \otimes(\Sym ^{c-1} \cF)^*\to {\bigwedge} ^{n} V \otimes (\Sym ^{c-2} \cF)^*&\to ... \to
 {\bigwedge} ^{n+3-c} V\otimes (\Sym ^1 \cF)^*\to
\\
\to
{\bigwedge} ^{n+2-c} V\otimes  \cO_{\PP (V)}\to \cO_{\PP (V)} (1+q+...+q^{n+1-c})&\to  \cO_{Z_c}(1+q+...+q^{n+1-c})\to 0.
\end{eqnarray*}
In particular, $\bigwedge ^{n+2-c}g$ induces an isomorphism
$$H^0 \left(\PP (V), I_{Z_c}\left(\frac{q^{n+2-c}-1}{q-1}\right)
\right)\simeq {\bigwedge} ^{n+2-c} V.$$
\end{corollary}

\medskip

By the modular interpretation of the Grassmannian, the map $g$, which
is surjective on $\PP (V)-Z_c$, defines a morphism $\varphi _c: \PP
(V)-Z_c\to \Gras _k\, (V, n+2-c)$. Composing this map with the
Pl\"ucker embedding $\Gras \, (V, n+2-c)\hookrightarrow \PP
({\bigwedge} ^{n+2-c} V)$, we get the rational map
$\PP(V)\dashrightarrow \PP ({\bigwedge} ^{n+2-c} V)$ defined by the
linear system $|I_{Z_c}(1+q+...+q^{n+1-c})|$.  This immediately
follows from Corollary \ref{Eagon-Northcott} and definition of the
Pl\"ucker embedding.  Note that $|I_{Z_c}(1+q+...+q^{n+1-c})|$ has
scheme theoretic base locus $Z_c$ and the map $\PP(V)\dashrightarrow
\PP ({\bigwedge} ^{n+2-c} V)$ is by Lemma \ref{blow-up} resolved by
the blow up $\Bl_{Z_c}\PP (V)\to \PP (V)$.  After fixing homogenous
coordinates in $\PP (V)$ this rational map is given by
$$[x_0,...,x_n]\longrightarrow [\Delta_{q}(x_{I})]_{\{I:\, |I|=n+2-c\}}.$$

We have canonical isomorphisms ${\bigwedge} ^{i}V\simeq {\bigwedge}
^{n+1-i} V^*\otimes \det V$ for $i=0,...,n+1$.  They induce canonical
isomorphisms $\PP({\bigwedge} ^{i}V)\simeq \PP({\bigwedge} ^{n+1-i}
V^*)$ that are independent of the choice of an orientation $ \det
V={\bigwedge} ^{n+1} V\simeq k$.  Similarly, we have canonical
isomorphisms $\Gras _k (V, i)\simeq \Gras _k (V^*, n+1-i)$ compatible
with the corresponding Pl\"ucker embeddings.  In fact, this
isomorphism can be easily described on the level of functors as
sending a quotient $V_S\to \cE $ to $V_S^*\to (\ker \, (V_S\to \cE))
^*$.

Therefore we can treat $\varphi _c$ as a morphism $ \PP (V)-Z_c\to
\Gras _k\, (V^*, c-1)$. Composing it with the Pl\"ucker embedding $
\Gras (V^*, c-1)\hookrightarrow \PP ({\bigwedge} ^{c-1} V^*)$ we again
obtain the map defined by the linear system
$|I_{Z_c}(1+q+...+q^{n+1-c})|$. In coordinates the corresponding
rational map is given by the same equations as before but now the map
is written as
$$[x_0,...,x_n]\longrightarrow [\Delta_{q}(x_{\hat I})]_{\{I:\, |I|=c-1\}}.$$
Note that for $c=2$ we get a rational map $\PP (V)\dashrightarrow \PP
(V^*)$ given by the linear system $|I_{Z_2}\left(
  \frac{q^n-1}{q-1}\right)|$. The study of this rational map is the
subject of the following sections.

\section{Linear system of  plane curves with moving singularities}

The following proposition generalizes to higher characteristic Serre's
example of a linear system with moving singularities in characteristic
$2$, described in one of exercises in Hartshorne's book (see
\cite[Chapter III, Exercise 10.7]{Har}). It is probably folklore known
to some experts or people trying to solve Hartshorne's exercise. It is
essentially equivalent to \cite[Theorem 4.1]{CT} (with a different
proof) but it was also known to the author earlier. We describe 
this example separately as it is an elementary special case
of the more general linear system used later in Section \ref{endo}.

\medskip

\begin{proposition}\label{Serre}
  The linear system $|I_Z(q+1)|$ of degree $(q+1)$ plane curves on
  $\PP^2$ containing $Z=\PP^2(\FF_q)$ has dimension $2$ with base
  points $Z$. It determines an inseparable degree $p$ morphism from
  $\PP^2-Z$ to $\PP^2$, which as a rational map $\PP^2\dashrightarrow
  \PP^2$ is resolved by blowing up $Z$.  Moreover, on $\PP^2_{\bar k}$
  every curve $C\in |I_Z(q+1)|$ is singular at exactly one point of
  $\PP^2(\bar k)$ determining a bijection between curves in
  $|I_Z(q+1)|$ and $\PP^2 (\bar k)$.
\end{proposition}

\begin{proof}
The map $\theta$ from Subsection \ref{Hirokado} gives rise to the short exact sequence
$$0\to \cO_{\PP^2}(-q+1)\mathop{\to}^{\theta} T_{\PP^2}\to  I_Z(q+2)
\to 0.$$
Twisting this sequence  by  $\cO_{\PP^2}(-1)$ we get
$$0\to \cO_{\PP ^2} (-q)\to T_{\PP^2}(-1)\to I_Z(q+1)
\to 0.$$
Therefore we get $h^0(I_Z(q+1))=3$. A basis of $H^0(\PP^2, I_Z(q+1))$ can be obtained
by fixing three distinct points of $Z$ and taking products of equations of all
$\FF_q$-rational lines passing through the fixed points. Let us recall that the
group $\FF_q^*$ is cyclic and hence
$$x_1\prod _{a\in \FF_q}(x_2-ax_1)=x_1x_2^q-x_1^qx_2$$
(this is also a special case of Lemma \ref{Moore}).
So for example taking
$[1,0,0]$, $[0,1,0]$ and $[0,0,1]$, we see that
$s_0=x_1x_2^q-x_1^qx_2$, $s_1=x_0^qx_2-x_0x_2^q$ and $s_2=x_0
x_1^q-x_0^qx_1$ form a basis of $H^0(\PP^2, I_Z(q+1))$.

Let us fix $ a,b,c \in \bar k$, not all equal to $0$, and consider the polynomial
$$f(x_0,x_1,x_2):=as_0+bs_1+cs_2.$$
Since
$$\frac{\partial f}{\partial x_0}=-bx_2^q+cx_1^q,\quad \frac{\partial f}{\partial x_1}=ax_2^q-cx_0^q,\quad
\frac{\partial f}{\partial x_2}=bx_0^q-ax_1^q,$$ the point
$P=[\sqrt[q]{a} ,\sqrt[q]{b},\sqrt[q]{c}]\in \PP^2(\bar k)$ is the only
singular point of the curve $f(x_0,x_1,x_2)=0$.

Assume that $a\ne 0$ and let us set
$$y_1=\frac{x_1}{x_0}-\sqrt[q]{\frac {b}{a}}\quad\hbox{ and }\quad y_2=\frac{x_2}{x_0}-\sqrt[q]{\frac {c}{a}}.$$
Then in the corresponding affine chart, the curve $f(x_0,x_1,x_2)=0$
has equation
$$ \left(\sqrt[q]{\frac{c}{a}} -{\frac{c}{a}}\right)y_1^q +\left(\sqrt[q]{\frac{b}{a}} -{\frac{b}{a}}\right)y_2^q  + y_1y_2^q-y_1^qy_2=0.$$
Therefore if $[a,b,c]\in \PP^2(\FF_q)$ then in this chart the curve $f(x_0,x_1,x_2)=0$ is given by $y_1y_2^q-y_1^qy_2=0$. So the
curve $f(x_0,x_1,x_2)=0$ consists of all $\FF_q$-rational lines passing through
the point $P$ and $P=[a,b,c]$. If $[a,b,c]\not \in \PP^2(\FF_q)$ then
$P$ is a singular point of $f(x_0,x_1,x_2)=0$ of multiplicity $q$ with
only $1$ tangent direction.

Let us consider the Frobenius morphism $\Fr_{\PP^2}:\PP^2\to
\PP^2$ given by $[x_0,x_1,x_2]\to [x_0^q,x_1^q,x_2^q]$. Note
that by definition $P$ is the preimage of $[a,b,c]$. If $P$ is a fixed
point of $\Fr_{\PP^2}$ then the curve $f(x_0,x_1,x_2)=0$ consists of $(q+1)$
lines passing through $P$. Otherwise, $f(x_0,x_1,x_2)=0$ has only one
singularity of multiplicity $q$ at $P$.
This shows that the correspondence sending $C\in |I_Z(q+1)|$ to the
singular point of $C$ is a bijection and the singular points of
elements of the linear system $|I_Z(q+1)|$ move all over $\PP ^2(\bar k)$.
\end{proof}

\begin{remark}
Note that a part of \cite[Chapter III, Exercise 10.7]{Har} is
  incorrect. Namely, the curves $f(x_0,x_1,x_2)=0$ with singularities
  outside of $Z$ can be reducible (even in characteristic $2$).

  For example, consider the curve $C:=(f(x_0,x_1,x_2)=0)$ with
  singularity at a point $Q\not \in Z$ but lying on a line $L$ defined
  over $\FF_q$.  Note that $C$ contains at least $(q+2)$ points of
  $L$. Namely, it contains point $Q$ and $(q+1)$ points $L(\FF_q)$.
  Since $C$ has degree $(q+1)$, it must contain $L$ and hence
  $D$ is reducible. More precisely $C=L+C'$, where $C'$ is a degree
  $q$ curve passing through $Z-L$ and having multiplicity $(q-1)$ at
  $Q$ (it is easy to see that at $Q$, the curve $C'$ is of type
  $\alpha y_1^{q-1}-y_1^{q-1}y_2+ y_2^q=0$ for some $\alpha\ne
  0$). Moreover, $C'$ intersects $L$ at exactly one point $Q$ as any
  other intersection point would give another singularity of $C$.  In
  characteristic $2$ the corresponding cubic curve $C$ consists of the
  line $L$ and a smooth conic $C'$ tangent to it at $Q$.
\end{remark}

\section{Purely inseparable endomorphisms of Drinfeld half-spaces} \label{endo}

As before, let $V$ be an $(n+1)$-dimensional vector space over a finite field
$k=\FF_q$, where $q=p^e$ for some prime number $p$ and a positive integer $e$.
Let $\Omega (V)$ be the Drinfeld half-space, i.e., the complement of
all $k$-rational hyperplanes in $\PP (V)$. In the following we assume
that $n\ge 1$ as everything is trivial if $\dim V\le 1$.

Let us consider Euler's exact sequence
$$0\to \cO_{\PP (V)}(-1)\mathop{\longrightarrow}^{s} V^*\otimes \cO_{\PP (V)}\to T_{\PP (V)}(-1)\to 0,$$
where $s$ is the dual of the evaluation map $V\otimes \cO_{\PP
(V)}\to \cO_{\PP (V)}(1)$.  This gives us a canonical embedding
$\PP(T_{\PP(V)})\subset \PP(V)\times_k\PP (V^*)$.  Pulling back
the above sequence by the composition of Frobenius endomorphisms
$\Fr ^i: \PP(V)\to \PP (V)$ we obtain
$$0\to \cO_{\PP (V)}(-q^i)\mathop{\longrightarrow}^{(\Fr ^i)^*s} V^*\otimes \cO_{\PP (V)}\to ((\Fr^i)^*T_{\PP (V)})(-q^i)\to 0.$$
So we also have canonical embeddings
$\PP((\Fr^i)^*T_{\PP(V)})\subset \PP(V)\times_k\PP (V^*)$.

Let $X_V$ be the scheme theoretic intersection of $
\bigcap_{i=0}^{n-1}\PP((\Fr ^i)^*T_{\PP(V)})\subset \PP(V)\times_k\PP
(V^*)$.  Projections of $X_V$ onto the first and second factor are
denoted by $\pi_{1,V}$ and $\pi_{2,V}$, respectively.

In the following we denote by $\cL (V)$ the set of all proper
$k$-linear subspaces of $\PP (V)$.  Every $k$-linear subspace
$L\in \cL (V)$ is of the form $\PP (V/W)$ for some $k$-linear
subspace $W\subset V$. For $L=\PP (V/W)\subsetneq \PP (V)$ we set
$L^{\perp}= \PP (W^*)\subsetneq\PP (V^*)$.  If $\eta _L$ denotes
the generic point of $L$ then we set
$E_L:=\overline{\pi_{1,V}^{-1}(\eta _L)}\subset X_V$.

\begin{lemma} \label{linear-images}
For any $L\in \cL (V)$ we have $E_L=L\times _k L^{\perp}$  and
$\pi _{1,V}^{-1}(L)=\bigcup_{L'\in \cL (V), L'\subset L} E_{L'}$. In particular,
$E_L/k$ is smooth of dimension $(n-1)$ and $\pi_{1,V}^{-1}(L)$ is pure of dimension $(n-1)$.
\end{lemma}

\begin{proof}
  In the following $[v]\in \PP (V)(\bar k)$ denotes the $\bar
  k$-rational point determined by $\bar k$-point $v\in V^*({\bar
    k})$. Similarly, $[w]\in \PP (V^*)(\bar k)$ denotes the $\bar
  k$-point determined by $w\in V({\bar k})=(V^*)^*({\bar k})$.  Let us
  note that $X_V({\bar k})\subset \PP (V)({\bar k})\times \PP
  (V^*)({\bar k})$ is equal to
  $$\{ ([v], [w]): \langle v, w\rangle=\langle\Fr (v), w\rangle=...=\langle\Fr ^{n-1}(v), w\rangle=0
\} ,$$ where $\langle\cdot, \cdot \rangle$ denotes the natural
pairing $V^*(\bar k)\times V({\bar k})\to \bar k$. $\bar
k$-rational points of $\PP (V)$ correspond to hyperplanes in
$V_{\bar k}$ and a subspace $W\subset V_{\bar k}$ is $k$-rational
if and only if $\Fr (W)=W$.  So for a $k$-rational subspace
$L\subset \PP (V)$ we have $([v], [w])\in \pi_{1,V} ^{-1}(L)({\bar
k})$ if $[v]\in L({\bar k})$ and $[w]\in L^{\perp}({\bar k})$.
Note that if $L\subsetneq \PP (V)$ has codimension $c$ then
 $L^{\perp}\subsetneq \PP (V^*)$ has dimension $(c-1)$.
Therefore if $\eta _L$ denotes the generic point of $L$ then
 $$E_L=\overline{\pi_{1,V}^{-1}(\eta _L)}= \overline{\pi_{1,V} ^{-1}(L-\bigcup_{L'\in \cL (V), L'\subsetneq L} L')}=L\times_k L^{\perp}
 $$
 is a smooth divisor in $X$ and $\pi _{2,V}(\pi_{1,V}^{-1}(\eta _L))$ is the
 generic point of the $k$-rational subspace $L^{\perp}\subsetneq \PP (V^*)
$.
\end{proof}

\medskip

From now on to simplify notation we denote $Z_2$ from Proposition \ref{codim-2} by $Z$.

\begin{proposition} \label{prop-1}
The $k$-scheme  $X_V$ has the following properties.
\begin{enumerate}
\item $X_V$ is geometrically integral.
\item $X_V$ has dimension $n$. In particular, $X_V$ is a complete intersection in  $ \PP(V)\times_k\PP
(V^*)$.
\item $\pi _{1,V}: X_V\to \PP (V)$ is the blow
  up of $\PP (V)$ along $Z$.
\item The singular locus of $X_V/k$ is a divisor equal to $\Sing
  X_V=\bigcup_{L\in \cL (V), \codim L \ge 3}E_L$.
\item  $X_V$ is smooth if $n=2$ and $X_V$ is non-normal if $n\ge 3$.
\item $X_V$ is the graph of the rational map $\psi _V: \PP (V)\dashrightarrow
  \PP (V^*)$.
\end{enumerate}
\end{proposition}

\begin{proof}
Let us fix homogeneous coordinates  $x_0,...,x_n$ in $ \PP(V)$ and
dual homogeneous coordinates $y_0,...,y_n$ in $\PP (V^*)$. Then
$X_V$ is defined by the ideal in the bigraded ring $k[x_0,...,x_n;
y_0,...,y_n]$ generated by bihomogeneous polynomials
$f_0=x_0y_0+...+x_ny_n$, $f_1=x_0^{q}y_0+...+x_n^{q}y_n$, ...,
$f_{n-1}=x_0^{q^{n-1}}y_0+...+x_n^{q^{n-1}}y_n$. To check where
$X_V\to \Spec k$ is smooth let us consider the Jacobian matrix
  $$\left(\frac{\partial f_i}{\partial x_j},\frac{\partial f_i}{\partial y_j}\right)_{i=0,...,n-1, j=0,1,...,n}=\left(
  \begin{array}{cccccc}
  y_0&...&y_n& x_0&...&x_n\\
  0&...&0&x_0^q&...&x_n^q\\
\vdots&&\vdots&\vdots&&\vdots\\
  0&...&0&x_0^{q^{n-1}}&...&x_n^{q^{n-1}}\\
  \end{array}
  \right) .
  $$
The only non-zero $n$th order minors  of this matrix (up to a sign) are of the form
 $$ \Delta_q (x_0,...,\hat x_i,..., x_n) \quad\hbox{and } \quad y_i\cdot (\Delta_q(x_0,...,\hat {x_i}, ...,
\hat {x_j},...,x_n))^q$$ for $i,j=1,...,n$ such that $i\ne j$. Let
$x$ be a point of the subscheme of $X_V$ defined by these minors.
Since for some $i$ we have $y_i(x)\ne 0$, $x$ lies in the preimage
under $\pi_{1,V}$ of the subscheme of $\PP (V)$ defined by the
ideal generated by $\Delta_q(x_0,...,\hat {x_i}, ..., \hat
{x_j},...,x_n)$ for $i\ne j$. By Proposition \ref{codim-2} we have
$x\in \pi_{1,V}^{-1}(\bigcup_{L\in \cL (V),  \codim L= 3}L)$.
Therefore the above Jacobian matrix has rank $n$ precisely at the
points not lying over a $k$-linear subspace of $\PP (V)$ of
codimension $3$. It follows that $\pi_{1,V}^{-1}(\PP
(V)-\bigcup_{L\in \cL (V),
  \codim L= 3}L)$ is smooth of dimension $n$ over $k$
  and the singular locus of $X_V/k$ equals to
$\pi_{1,V}^{-1}(\bigcup_{L\in \cL (V), \codim L= 3}L)=\bigcup_{L\in \cL (V), \codim L \ge 3}E_L$.

\medskip
Expansion with respect to the last row gives
$$0=\det
\left(
\begin{array}{ccc}
x_0&...&x_n\\
\vdots&& \vdots\\
x_0^{q^{n-1}}&...&x_n^{q^{n-1}}\\
x_0^{q^{j}}&...&x_n^{q^{j}}\\
\end{array}
\right) =\sum (-1)^{n+i}x_i^{q^j}\Delta_{q}(x_0,...,\hat{x_i},
...,x_n)$$ for $j=0,...,n-1$. Therefore
$y_i-(-1)^{n+i}\Delta_{q}(x_0,...,\hat{x_i}, ...,x_n)\in I(X)$ for $i=0,...,n$
proving that the projection $\pi _{1,V}: X_V\to \PP (V)$ is surjective. Since the fibers of
$(\pi_{1,V})_{\bar k}: (X_V)_{\bar k}\to \PP (V)_{\bar k}$ are linear subspaces of $\PP (V)_{\bar k}$, it follows that
$X_V$ is geometrically connected.

\medskip

Now we claim that $\pi_{1,V}':=\pi_{1,V}|_{X_V-\pi_{1,V} ^{-1}(Z)}: X_V-\pi_{1,V}
^{-1}(Z)\to \PP (V)-Z$ is an isomorphism. To prove that it is
sufficient to check that $\pi_{1,V}'$ is smooth and bijective on points.

Let us take $x\in \PP (V)-Z$. Then there exists $i$ such that $x\in D_{+}(\Delta_q(x_0,...,\hat x_i ,...,x_n))$.
We claim that $\pi_{1,V}^{-1}(x) \subset \pi_{2,V}^{-1}(D_{+}(y_i))$.
To prove that it is sufficient to note that the system of linear equations (in $y_0,...,\hat y_i,...y_n$)
$$\left\{
\begin{array} {cccccccccl}
  x_0y_0&+&...&+&\widehat{x_iy_i}&...&+&x_ny_n&=&0\\
  x_0^qy_0&+&...&+&\widehat{x_i^qy_i}&...&+&x_ny_n&=&0\\
  \vdots&& &&\vdots&&&\vdots &&\vdots\\
  x_0^{q^{n-1}}y_0&+&...&+&\widehat{x_i^{q^{n-1}}y_i}&...&+&x_ny_n&=&0
\end{array} \right.
$$
has no non-zero solutions.
If we consider $\pi_{2,V}^{-1}(D_{+}(y_i))$ then $\pi_{1,V}^{-1}(x)$ is defined  by
the system of linear equations
$$\left\{
\begin{array} {ccccccl}
x_1u_1&+&...&+&x_nu_n&=&-x_0\\
x_1^qu_1&+&...&+&x_n^qu_n&=&-x_0^q\\
\vdots&& &&\vdots&&\vdots\\
x_1^{q^{n-1}}u_1&+&...&+&x_n^{q^{n-1}}u_n&=&-x_0^{q^n}
\end{array} \right.
$$
which has exactly one solution.
Corollary \ref{codim-2} implies that the union of all $k$-linear
subspaces of codimension $2$ is given by vanishing of
$\Delta_{q}(x_0,...,\hat{x_i}, ...,x_n)$ for $i=0,...,n$.  Therefore
the restriction of $\pi_1$ to the preimage of the union of all
$k$-linear subspaces of codimension $2$ is an isomorphism with the
inverse given by $x\to (x, \psi (x))$, where
$\psi : \PP (V)-Z \to \PP (V^*)=\Proj k[y_0,...,y_n]$ is given by
$$\psi ^* (y_i)=(-1)^{i}\Delta_{q}(x_0,...,\hat{x_i}, ...,x_n).$$
In particular, by Section \ref{determinantal} the map $\psi$ is given by the linear system
$|I_Z(1+q+...+q^{n-1})|$.

Since the ideal of $X_V$ is generated by $n$ elements, each irreducible component of $X_V$ has dimension
$\ge \dim \PP (V)\times _k \PP (V^*)- n=n$. On the other hand,
$X_V-\pi_{1,V} ^{-1}(Z)$ is irreducible and by Lemma \ref{linear-images}
 $\pi_{1,V} ^{-1}(Z)$ is pure of dimension $n-1$, so
$X_V$ is irreducible.  In fact, the same arguments work also for $(X_V)_{\bar k}$, so $X_V$ is geometrically
irreducible. Since $X_V$ is smooth at the generic point and it is locally a complete intersection, it is also
geometrically reduced.
Therefore $X_V$ is geometrically integral and by the above it is equal to the graph of
$\psi: \PP (V) \dashrightarrow \PP (V^*)$.
Then Lemma \ref{blow-up} implies that $X_V$ is the blow up of $\PP (V)$ along $Z$.
\end{proof}

\begin{proposition} \label{prop-2}
There exists a finite, purely inseparable morphism $\varphi_{X_V}: X_V\to X_{V^*}$
such that the diagram
$$ \xymatrix{
X_V\ar[d]^{\pi_{1,V}} \ar[r]^{\varphi_{X_V}}\ar[rd]^{\pi_{2,V}}&X_{V^*}
\ar[d]^{\pi_{1,V^*}}\\
\PP (V)\ar@{-->}[r]^{\psi _V}& \PP (V^*)\\
}$$
is commutative. Moreover, we have $\varphi_{X_{V^*}}\circ \varphi_{X_V}=\Fr _{X_V} ^{n-1}$.
The restriction of $\psi _V$ to $\Omega (V)$ gives a finite purely inseparable morphism $\psi _V: \Omega(V)\to \Omega (V^*)$
of degree  $q^{\binom{n}{2}}$. Thus after fixing an isomorphism $\Omega(V^*)\simeq \Omega (V)$, we get a purely inseparable
endomorphism of Drinfeld's half-space $\Omega (V)$.
\end{proposition}

\begin{proof}
Let  $p_1$ and $p_2$ be the projections of $\PP (V)\times _k \PP (V^*)$ onto the first and second factor, respectively.
Let us consider the map $\bar \varphi _V: \PP (V)\times _k \PP (V^*) \to \PP (V^*)\times _k \PP (V)$ defined by
$$\bar \varphi_V=   p_2\times _k (\Fr_{\PP (V)}^{n-1}  \circ p_1).$$
After fixing projective coordinates
$\bar \varphi _V$ is given by
$$([x_0,...,x_n], [y_0,...,y_n]) \to ([v_0,...,v_n], [w_0,...,w_n])=([y_0,...,y_n], [x_0^{q^{n-1}},...,x_n^{q^{n-1}}]).
$$
Since
$$(\sum x_i^{q^j}y_i)^{q^{n-1-j}}=\sum x_i^{q^{n-1}}y_i^{q^{n-1-j}}=\sum v_i^{q^{n-1-j}}w_i$$
for $j=0,1,...,n-1$,  $\bar \varphi _V$ maps $\PP((\Fr ^j)^*T_{\PP(V)})$ into $\PP((\Fr ^{n-1-j})^*T_{\PP(V^*)})$
and therefore  $X_V$ into $X_{V^*}$. We define $\varphi_{X_V}$ as
the restriction $\bar\varphi_V |_X:X_V\to X_{V^*}$.

Let us note that $\pi_{2, V}=\pi_{1, V^*}\circ \varphi_{X_V}$.
Moreover, since $\bar\varphi_{V^*}\circ \bar\varphi_{V}=\Fr_{\PP (V)\times _k \PP (V)}^{n-1}$,
we have $\varphi_{X_{V^*}}\circ \varphi_{X_V}=\Fr _{X_V} ^{n-1}$. In particular, $\varphi_{X_V}$ is a purely inseparable map of degree $q^{\binom{n}{2}}$.

We need to prove that $\psi_V$ maps $\Omega(V)$ into $\Omega (V^*)$.
Since $\pi_{1,V}:X_V \to \PP (V)$ is the blow up of $\PP (V)$ along $Z$ and Lemma
\ref{Moore} implies that $\Omega (V)$ is contained in
$\PP (V)-Z$, $\pi_{1,V}|_{\pi_{1,V}^{-1}(\Omega (V))}:\pi_{1,V}^{-1}(\Omega
(V)) \to \Omega (V)$ is an isomorphism. So we can treat $X_V$ as a
compactification of Drinfeld's half-space $\Omega (V)$. Since by Lemma \ref{linear-images}
$$\pi_{1,V}^{-1}(\Omega (V))=X_V-\bigcup _{{L\in \cL(V)}}E_L $$
and $\varphi_{X_V}$ maps $E_{L}$ into $E_{L^{\perp}}$, $\varphi_{X_V}$
maps $\pi_{1,V}^{-1}(\Omega (V))$ into $\pi_{1,V^*}^{-1}(\Omega (V^*))$.  But $\psi_V
|_{\PP (V)-Z}= \pi_{1, V^*} \circ\varphi_{X_V} \circ (\pi_{1,V}|_{\PP (V)-Z})^{-1}$,
so $\psi_V (\Omega (V))\subset \Omega (V^*)$.
Since $\varphi_{X_V}$ is finite and purely inseparable,  $\psi _V: \Omega(V)\to \Omega (V^*)$ is also finite
and purely inseparable. Moreover, since  $\varphi_{X_{V^*}}\circ \varphi_{X_V}=\Fr _X ^{n-1}$ we have $\psi_{V^*}\circ \psi_V=\Fr_{\Omega(V)}^{n-1}$.
\end{proof}

\begin{proposition} \label{prop-3}
The projection $\pi_{2,V}:X_V\to \PP (V^*)$ onto the second factor decomposes into a birational morphism $f_V:X_V\to Y_V$  and a purely inseparable morphism $\varphi _{Y_V}: Y_V\to \PP (V^*)$.
Morphism $f_V$ is an isomorphism outside of $(\bigcup _{{L\in \cL(V)},\dim L >0}E_L)$ and we have
$\codim E_L=1$ and $\dim f(E_L)=\codim L-1.$
\end{proposition}

\begin{proof}
  Let us consider Stein's factorization $X_V\to Y_V\to \PP (V^*)$ of $\pi_{2,V}$.
 In particular, $f_V:X_V\to Y_V$ is birational with connected fibers and $\varphi _{Y_V}: Y_V\to \PP (V)$ is
  finite. Note that $$\pi_{2,V}|_{\pi _{1,V}^{-1}(\Omega(V))}: \pi _{1,V}^{-1}(\Omega(V)) \mathop{\to}^{\simeq} \Omega(V)\mathop{\to}^{\psi_V}\Omega(V^*)$$ is finite and purely inseparable, so $\varphi_{Y_V}$ is purely inseparable.

Clearly, we have $\varphi_{X_V}(E_L)=E_{L^{\perp}}$. Since $E_L$ is defined over $k$, we have $\Fr_X^{-1}(E_L)=E_L$
and hence $\varphi_{X_V}^{-1}(E_L)=E_L$.
Note that
$$\pi_{2,V}^{-1}(Z)=\varphi_{X_V}^{-1}(\pi_{1,V}^{-1}(Z))=\varphi_{X_V}^{-1} ( \bigcup _{{L\in \cL(V)},{\codim L \ge 2}}E_L)= \bigcup _{{L\in \cL(V)},\codim L \ge 2}
E_{L^{\perp}}= \bigcup _{{L\in \cL(V)}, \dim L >0}E_L$$ and $f_V$ is an
isomorphism on $\pi_{2,V}^{-1}(\PP (V)-Z)$. Therefore $\Exc (f_V)=\bigcup
_{{L\in \cL(V)},\dim L >0}E_L$.
\end{proof}

\begin{remark}
It is easy to see that $Y_V$ is a normal projective $k$-variety. Indeed, there exists a birational morphism
from the wonderful compactification $\bar X\to X_V$  (cf. Section \ref{wonder-section}) and $\bar X\to Y_V\to \PP(V^*)$ is Stein's factorization, so $Y_V$ is normal.
\end{remark}

\begin{theorem}  \label{main-3'}
There exists a finite morphism $\varphi _V: \PP (V)\to Y_{V^*}$ such that
we have the following commutative diagram of $k$-schemes
$$ \xymatrix{
Y_V\ar[d]^{\varphi_{Y_V}}&X_V\ar[l]_{f_V}\ar[ld]^{\pi_{2,V}}
\ar[r]^{\pi_{1,V}}\ar[d]^{\varphi_{X_V}}&\PP (V)\ar[d]^{\varphi _V}\\
\PP (V^*)&X_{V^*}\ar[l]^{\pi_{1,V^*}}\ar[r]_{f_{V^*}}&Y_{V^*}\\
}$$
in which horizontal maps are birational, vertical maps are purely inseparable
and they satisfy the following relations
$$\varphi_{X_{V^*}}\circ \varphi_{X_V}=\Fr _{X_V}^{\dim V-2},\quad \varphi _{Y_{V^*}}\circ \varphi_{V} =\Fr_{\PP(V)}^{\dim V-2}, \quad \varphi_{V^*}\circ \varphi _{Y_V}=\Fr_{Y_V}^{\dim V-2}.$$
The rational map  $\psi _V:=\pi_{1, V^*}\varphi_{X_V}\pi_{1,V}^{-1}: \PP(V)\dashrightarrow \PP (V^*)$ is  purely inseparable and it satisfies $\psi_{V^*}\circ \psi _V=\Fr_{\PP (V)}^{\dim V-2}$. Moreover, $\psi_V$ defines a finite morphism $\psi_V: \Omega(V)\to \Omega (V^*)$ such that $\psi_{V^*}\circ \psi _V=\Fr_{\Omega (V)}^{\dim V-2}$.
\end{theorem}

\begin{proof}
  Let us consider Stein's factorization of the map $f_{V^*}\circ
  \varphi_{X_V}: X_V\to Y_{V^*}$.  We get a proper morphism
  $g_1:X_V\to Y'$ such that $(g_1)_*\cO_{X_V}=\cO_{Y'}$ and a finite
  affine morphism $g_2:Y'\to Y_{V^*}$. Now let us consider the
  morphism $h:X_V\to \PP (V)$ defined by $h:=\Fr_{\PP (V)}^{\dim
    V-2}\circ \pi_{1,V}$. Note that $h_2=\varphi_{Y_{V^*}}\circ g_2$
  is finite and affine, so we have two Stein's factorizations of $h$,
  one coming from the definition and another one given by $h=h_2\circ
  g_1$. It follows that $Y'$ and $\PP (V)$ are isomorphic as affine
  schemes over $\PP (V)$ (as they are spectrums of the same sheaf of
  algebras $h_*\cO_{\PP (V)}$) and the birational map $g_1\circ
  \pi_{1,V}^{-1}:\PP (V)\to Y'$ extends to an isomorphism. It follows
  that the birational map $\varphi_V:= f_{V^*}\circ \varphi_{X_V}\circ
  \pi_{1,V}^{-1}: \PP (V)\dashrightarrow Y_{V^*}$ extends to a morphism.

The remaining assertions follow easily from Propositions
\ref{prop-1}, \ref{prop-2} and \ref{prop-3}.
\end{proof}

\section{Wonderful compactification of Drinfeld's half-space}\label{wonder-section}

As before, let $V$ be an $(n+1)$-dimensional vector space over a finite field
$k=\FF_q$ and let $\cL (V)$ be the set of all proper $k$-linear subspaces of $\PP (V)$.
$\cL (V)$ forms an arrangement of subvarieties of $\PP (V)$ and so we can consider
the wonderful compactification $\tilde X$ of $\Omega(V)$. By Theorem \ref{wonderful-comp}
we can construct $\tilde X$ by successively blowing up strict transforms of $k$-linear subspaces of
$\PP(V)$ of increasing dimensions. So we have
$$\tilde X=\tilde X_{n}\mathop{\longrightarrow}^{\tilde \pi_{n-1}} \tilde  X_{n-1}\longrightarrow \cdots \longrightarrow \tilde X_2\mathop{\longrightarrow}^{\tilde \pi_{1}} \tilde  X_{1}\mathop{\longrightarrow}^{\tilde \pi_{0}}  \tilde X_{0}=\PP(V),$$
where $\tilde  \pi_i$ is the blow up of $\tilde X_{i}$ along the  strict transforms of $k$-linear subspaces of
$\PP(V)$ of dimension $i$ (in particular $\tilde \pi_{n-1}$ is an isomorphism). The canonical map $\tilde  X\to \PP (V)$ is denoted by $\tilde \pi$. The pull-back of the hyperplane section to $\tilde  X$ is denoted by $H$.

Note that by Lemma \ref{blow-up-sum} we have canonical morphisms $\tilde  X_i\to \Bl_{Z_c}\PP (V)$ for $c=n+1-i,...,n+1$.
Therefore Corollary \ref{Eagon-Northcott} implies that  for $c=n+1-i,...,n+1$ the linear system $|I_{Z_c}(1+q+...+q^{n+1-c})|$ becomes base point free on $\tilde X_i$ giving rise to the canonical morphism
$$\varphi _{c}^{(i)}: \tilde  X_i\to \PP ({\bigwedge} ^{n+2-c} V).$$
For $c=n+1$ we have $Z_c=\emptyset$ and $|I_{Z_c}(1+q+...+q^{n+1-c})|=|\cO_{\PP (V)}(1)|$, so $\varphi^{(i)}_{ n+1}=\tilde \pi_0\circ \dots \circ \tilde  \pi_{i-1}$.
For $c=2,...,n+1$ we set $\varphi_c=\varphi ^{(n)}_{c}$. Note that for  $c=1$ we have  $I_{Z_c}(1+q+...+q^{n+1-c})=\cO_{\PP (V)}$, so  we can also define $\varphi_1$ as the map of $\tilde X$ to  $\PP ({\bigwedge} ^{n+1} V)$.

\begin{lemma}\label{log-general-type}
Let us fix  $0\le i\le n$.  The morphism
$$\varphi _{\tilde  X_i}: \tilde  X_i\to \PP (V)\times _k\PP ({\bigwedge} ^2 V)\times_k ... \times_k \PP ({\bigwedge} ^{i+1} V)$$
given by $(\varphi_{ n+1}^{(i)}, \varphi_{ n}^{(i)},..., \varphi_{n+1-i}^{(i)})$
is a closed embedding.
\end{lemma}

\begin{proof}
The proof is by induction on $i$, with $i=0$ being trivial. So let us assume that we know the assertion for $i-1$.
By the construction and Lemma \ref{blow-up} we get a closed embedding $ \tilde  X_i\hookrightarrow  \tilde  X_{i-1}\times_k \PP ({\bigwedge} ^{i+1} V)$, so the lemma follows from the induction assumption.
\end{proof}

\medskip

Let us set
$$ D_V^i:=\sum_{L\in \cL(V),\, \codim L=i} D_{L}
$$
for $i=1,2,...,n$ (notation for $D_L$ is explained in Subsection \ref{wonderful-sec}). Let us also set
$D:=\sum_{i=1}^nD_V^i.$
The exceptional divisor of  $\tilde \pi: \tilde  X\to \PP (V)$ is equal to the sum of all $D_V^i$ for $i=2,...,n$. The following lemma follows by a straightforward computation:

\begin{lemma}
Let us set 
$$H_{c}:= \frac{q^{n+2-c}-1}{q-1}H- \sum _{i=0}^{n-c} \frac{q^{i+1}-1}{q-1}D^{i+c}_V .$$
Then we have $\cO_{\tilde X}(H_c)\simeq \varphi _{c}^*\cO_{\PP ({\bigwedge} ^{n+2-c} V)}(1)$ for $c=1,...,n+1$. 
Moreover, we have
$$K_{\tilde  X}+\sum _{j=1}^n D^j_V= (q-1)\sum_{c=1}^{n} H_{c+1}.$$
\end{lemma}

\begin{remark}\label{very-ample}
\begin{enumerate}
\item
In cases $c=n+1$ and $c=1$ the above lemma says that $H_{n+1}\sim H$ and
$H_1\sim 0$, so $\left(\sum_{i=0}^{n}q^i\right)H\sim  \sum _{i=0}^{n-1}
\left(\sum_{j=0}^{i}q^j\right)D^{i+1}_V.$
\item
Combining the above lemma
with Lemma \ref{log-general-type} we get very ampleness of $K_{\tilde
  X}+\sum _{j=1}^n D^j_V$. Note that in \cite[p.~227, Lemma]{Mu}
Mustafin claims that $K_{\tilde X}+\sum _{j=1}^n D^j_V$ is ample, but
his proof shows only that it is strictly nef.
\end{enumerate}
\end{remark}

\begin{lemma}
Let $\tilde Y$ be a smooth projective variety defined over an algebraically closed field and
let $B$ be a simple normal crossing divisor. Let us set $Y=\tilde Y-B$ and assume that $K_{\tilde Y}+B$
is ample. Then any separable and dominant rational map $Y\dashrightarrow Y$ extends to
an automorphism $\tilde Y\to \tilde Y$.
\end{lemma}

\begin{proof}
 The idea of proof is the same as that of \cite[Theorem 11.6]{Ii}, which is a general result
but depending on the characteristic zero assumption and using resolution of singularities.

Let $\varphi: Y\dashrightarrow Y$ be a dominant separable rational map. By the valuative criterion of properness,
the induced rational map $\tilde Y \dashrightarrow \tilde Y$ is defined at every codimension one point.
So there exists a closed subset $Z\subset \tilde Y$ of codimension at least $2$ such that $\varphi$ extends to a morphism
$\tilde \varphi : \tilde Y -Z \to \tilde Y$. The same computation as that in \cite[11.4 a]{Ii}
shows that $\tilde \varphi^* \cO (K_{\tilde Y}+ B)\subset \cO ( K_{\tilde Y} +\log B)|_{\tilde Y-Z}$ (here we use that $\varphi$ is dominant
and separable).
Since on a normal variety sections of a locally free sheaf extend outside of codimension $2$, $\tilde \varphi$ induces injective linear
maps $$\tilde \varphi _m ^*: H^0(\tilde Y, \cO_{\tilde Y} (m(K_{\tilde Y}+B)))\to
H^0(\tilde Y-Z, \cO_{\tilde Y} (m(K_{\tilde Y}+B)))=H^0(\tilde Y, \cO_{\tilde Y} (m(K_{\tilde Y}+B))).$$
Since the dimensions of both spaces are the same, $\tilde \varphi _m ^*$ is an isomorphism.
Now taking $m$ such that $m(K_{\tilde Y}+B)$ is very ample, the required assertion follows from the commutative diagram:
$$
\xymatrix{  \tilde Y\ar@{^{(}->}[d] \ar@{-->}[r]^{\varphi} & \tilde Y\ar@{^{(}->}[d] \\
\PP (H^0(m(K_{\tilde Y}+B))^*) \ar[r] ^{\simeq} &\PP (H^0(m(K_{\tilde Y}+B))^*) \\
}
$$
\end{proof}

\begin{corollary}
Every separable and dominant $k$-endomorphism of $\Omega(V)$ is a $k$-automorphism
and in particular it extends to a $k$-authomorphism of $\PP(V)$.
\end{corollary}

\begin{proof}
By Remark \ref{very-ample}.2 we know that  $K_{\tilde
  X}+\sum _{j=1}^n D^j_V$ is very ample. Since $\Omega(V)= {\tilde
  X}-\sum _{j=1}^n D^j_V$,  the previous lemma implies that a separable
and dominant $k$-endomorphism of $\Omega(V)$ extends to an automorphism of the wonderful compactification $\tilde X$. So the corollary follows as in  \cite[p.~1222]{RTW}.
\end{proof}

If $n=2$ then the wonderful compactification of $\Omega(V)$ is the
blow up of $\PP (V)$ along all $k$-rational points, so it coincides
with $X_V$. But if $n\ge 3$ then $X_V$ is non-normal, so it is not isomorphic
to the wonderful compactification of $\Omega(V)$.

\section{Relation to Deligne--Lusztig schemes} \label{DL-section}

The aim of this section is to provide an application of previous
results to the study of Deligne--Lusztig varieties corresponding to a
Coxeter element in the $A_n$ case. In particular, we prove that the
closure of the open Deligne-Lusztig variety in the full flag variety
is smooth (see Corollary \ref{smooth-DL}).

\medskip

Let $G$ be a connected reductive algebraic group defined over $\bar k$
and obtained by extension of scalars from $G_0$ defined over
$k=\FF_q$. Let us fix a $\Fr _G$-stable Borel subgroup $B\subset G$
containing a $\Fr_G$-stable maximal torus $T$. Let $X_G$ denote the
variety of Borel subgroups of $G$. The group $G$ acts on $X_G$ by
conjugation and there is a natural isomorphism $G/B\to X_G$ given by
$gB\to gBg^{-1}$.

Let $P$ and $Q$ be parabolic subgroups of $G$ containing $B$.
The product $G/P\times_{\bar k}G/Q$ with diagonal action of $G$ is a
$G$-variety. This $G$-variety is $G$-equivariantly isomorphic with
$G\times_P(G/Q)$, where the isomorphism
$$\xi: G\times_P(G/Q)\to G/P\times_{\bar k}G/Q$$
is given by $(g,hQ)\to (gP,ghQ)$.

Let $W=N(T)/T$ be the Weyl group of $G$. For any element $w\in W$ we define
the \emph{Bruhat cell} $C_{w,P}:=BwP\subset G/P$ and the \emph{Schubert
  variety} $S_{w,P}=\overline{C_{w,P}}\subset G/P$.

Let $\cO (w)\subset X_G\times_{\bar k} X_G $ be the $G$-orbit of $(eB,
\dot{w}B)$, where $\dot{w}$ is a representative of $w$ in $N(T)$.
It is equal to the image of $\xi (G\times_B C_{w,B})$. Another way to define it is to
say that $\cO (w)$ is the preimage of $w$ under the map
$$X_G\times_{\bar k} X_G\to G\backslash(X_G\times_{\bar k} X_G)=B\backslash G/B=W.$$
The Zariski closure $\overline{\cO(w)}\subset X_G \times_{\bar k}X_G$  is equal to the image of $\xi (G\times_B
S_{w,B})$.  Any closed irreducible $G$-stable subset of $G/P\times_{\bar
  k}G/Q$ is the image of some $\overline{\cO (w)}$ under the
projection $ X_G \times_{\bar k}X_G\to G/P\times_{\bar k}G/Q$.  The
image of $\overline{\cO (w)}$ under this projection is called a
\emph{$G$-Schubert variety} and denoted by $S_{w,P,Q}$.

If $w=s_1...s_n$ is a minimal expression for $w\in W$ then  $\overline{\cO (w)}$ has
Bott--Samelson (--Demazure--Hansen) desingularization
$\bar{\cO} (s_1,...,s_n)\to \overline{\cO (w)}$.

Let  $\Gamma\subset X_G\times_{\bar k} X_G$ be the graph of the Frobenius endomorphism
$\Fr_{X_G}$. The intersection of $\Gamma$ and $\cO (w)$ is transversal. We denote this intersection by
$X_G (w)$ and we call it  \emph{Deligne--Lusztig scheme}.  This scheme is obtained by extension
of scalars from a naturally defined $k$-scheme that we also denote by $X_G (w)$.
Note that traditionally $X_G(w)$ is called ``Deligne--Lusztig variety''. This name is rather unfortunate, as usually this
scheme is not a variety (often it is not  irreducible). Therefore we prefer to use a slightly different name.
By \cite[Lemma 9.11]{DL}  the graph $\Gamma$ is transverse to $\bar{\cO} (s_1,...,s_n)$ and the fibre product
$\overline{X_G} (s_1,...,s_n)$ is a smooth compactification of $X_G (w)$ with
complement being a normal crossing divisor.
We also have a canonical  map
$$\overline{X_G} (s_1,...,s_n)\to \overline{X_G (w)}\subset X_G.$$
Note that the closure of $X_G(w)$ in $X_G$ can be easily described using the Bruhat order in $W$
as
$$ \overline{X_G (w)} =\bigcup _{w'\le w}X_G(w') .$$

Let $P\subset G$ be a $\Fr _G$-stable parabolic subgroup containing $B$. If we set $W_P=(N(T)\cap P)/T$
then to any element $\bar w\in W_P\backslash W/W_P$ we can associate $\cO _P (\bar w)$ defined as
the preimage of $\bar w$ under the map
$$ G/P\times_{\bar k}G/P \to G\backslash( G/P\times_{\bar k}G/P)=P\backslash G/P=W_P\backslash W/W_P.$$
Then we define the \emph{generalized Deligne--Lusztig scheme} $X_{G,P}(\bar w)$  as
the product of the graph  $\Gamma _P \to  G/P\times_{\bar k}G/P$ of the Frobenius endomorphism  $\Fr_{G/P}$ and
$\cO _P (\bar w)\to  G/P\times_{\bar k}G/P$. If $w\in W$ is some lift of $\bar w$, then
$X_{G,P}(\bar w)$ can be also recovered as the image of
$X_G(w)$ under the canonical projection $G/B\to G/P$. If $\bar w$ is the class of $w$ we often write $X_{G,P}( w)$
instead of $X_{G,P}(\bar w)$.

\medskip

Now let $V$ be a $k$-vector space of dimension $(n+1)$ and let us consider the case when $G_0=\GL (V)$ and
$G=\GL (V_{\bar k})$. Then $W\simeq S_{n+1}$ and we consider the standard Coxeter element $w=(1,2,...,n+1)\in S_{n+1}$. The corresponding Deligne--Lusztig variety $X_G(w)$ is the Drinfeld's half-space $\Omega(V)$.
Let us fix a full flag of $k$-vector spaces
$$ V\twoheadrightarrow V_{n}\twoheadrightarrow ... \twoheadrightarrow V_2\twoheadrightarrow V_{1} ,$$
where  $\dim _{k}V_j=j$. Then we have  a standard  Borel subgroup  $B_0\subset G_0$ corresponding to linear
maps preserving this flag. Let us consider a sequence of parabolic groups $P_n=B_0\subsetneq P_{n-1}\subsetneq ...
\subsetneq P_1\subsetneq \GL (V)$ so that $P_i$ corresponds to linear maps preserving a partial flag of the form
$$ V\twoheadrightarrow V_{i}\twoheadrightarrow ... \twoheadrightarrow V_2\twoheadrightarrow V_{1} .$$
This  induces a sequence of flag varieties
$$\GL (V)/B_0\to\GL (V)/P_{n-1} \to ...\to \GL (V)/P_1= \PP (V).$$
Note that $\GL (V)/P_i\simeq \Flag (V;1,...,i)$, so for each $i$ we have the canonical embedding
$$\GL (V)/P_i\hookrightarrow
\Gras (V,1 )\times _k\Gras (V,2)\times_k ... \times_k \Gras (V, i)\hookrightarrow
\PP (V)\times _k\PP ({\bigwedge} ^2 V)\times_k ... \times_k \PP ({\bigwedge} ^{i} V),$$
where the first map is obtained by sending an $S$-point of $ \Flag (V;1,...,i)$
corresponding to a partial flag
$$ V_{S}\twoheadrightarrow \cE_{i}\twoheadrightarrow ... \twoheadrightarrow \cE_2\twoheadrightarrow \cE_{1} $$
to a tuple  $(V_S\twoheadrightarrow \cE_{j})_{j=1,...,i}$
(cf. proof of Theorem \ref{flag}) and the second map is the product of Pl\"ucker embeddings, given by sending the corresponding
tuple to
$$\left( V_S\twoheadrightarrow \det (\cE_1),{\bigwedge}^2 V_{S}\twoheadrightarrow \det (\cE_2), ..., {\bigwedge}^{i}V_{S}
\twoheadrightarrow \det (\cE_{i}) \right).$$

\begin{proposition}
We have a canonical isomorphism of the sequence  of closures of generalized Deligne--Lusztig $k$-schemes in flag varieties
$$\overline{X_G(w)}\to \overline{X_{G,P_{n-1}}(w)}\to ...\to \overline{X_{G,P_1}(w)} =\PP (V)$$
with the sequence
$$\tilde X=\tilde X_{n-1}\mathop{\longrightarrow}^{\tilde \pi_{n-2}} \tilde  X_{n-2}\longrightarrow \cdots \longrightarrow \tilde X_2\mathop{\longrightarrow}^{\tilde \pi_{1}} \tilde  X_{1}\mathop{\longrightarrow}^{\tilde \pi_{0}}  \tilde X_{0}=\PP(V)$$
considered in Section \ref{wonder-section}.
\end{proposition}

\begin{proof}
  It is well-known that $X_G(w)\subset \GL (V)/B_0$ is mapped
  isomorphically onto $\Omega (V)\subset \PP (V)$ (see
  \cite[2.2]{DL}), so the sequence of flag varieties induces
  isomorphisms of the corresponding generalized Deligne--Lusztig
  $k$-schemes $X_{G,P_i}(w)$. Since by Lemma \ref{log-general-type}
 $$\varphi _{\tilde  X_{i-1}}: \tilde  X_{i-1}\to \PP (V)\times _k\PP ({\bigwedge} ^2 V)\times_k ... \times_k \PP ({\bigwedge} ^{i} V)$$
is a closed embedding and it coincides with the composition
$$X_{G,P_i}(w)\subset \GL (V)/P_i\hookrightarrow
\PP (V)\times _k\PP ({\bigwedge} ^2 V)\times_k ... \times_k \PP ({\bigwedge} ^{i} V)$$
on the pre-image of $\Omega (V)$, we get the required assertion.
\end{proof}

\begin{corollary} \label{smooth-DL}
All varieties $\overline{X_{G,P_{i}}(w)}$ are smooth. In particular, 
the canonical map $\overline{X_G} (w)\to \overline{X_G (w)}$ is an isomorphism.
\end{corollary}

By \cite[Lemma 1]{Han} it was known that the map $\overline{X_G} (w)\to
\overline{X_G (w)}$ is bijective, but it seems that the above
corollary is new.

\section{Modular interpretation}

In this section we give a modular interpretation of all varieties and
maps defined in Section \ref{wonder-section} and we give the
corresponding interpretation of the morphisms from Theorem \ref{main-3}.

\subsection{Modular interpretation of flag varieties}

It is a standard fact that the Grassmannian is a fine moduli space for
the Grassman functor of quotient modules (or vector subbundles in case
of locally free sheaves).  More precisely, let $\cE$ be a coherent
$\cO_S$-module on a scheme $S$. A \emph{quotient module} of $\cE$ is
an equivalence class of surjective maps $q:\cE\to \cE'$ of coherent
$\cO_S$-modules such that two maps $q_1:\cE\to \cE'_1$ and $q_2:\cE\to
\cE'_2$ are equivalent if their kernels give the same subsheaf of
$\cE$. Let $\underline{\Gras} (\cE, r): (\Sch /S)^o\to \mathop{\rm
  Sets}$ denote the functor associating to an $S$-scheme $T$ the set
of all locally free quotient modules $\cE_T:=\cE \otimes _{\cO_S}\cO_T\to \cE
'$ of rank $r$. This functor is represented by a projective $S$-scheme
denoted by $\Gras _S(\cE,r)$ (see, e.g., \cite[Lecture 5]{Mum} or
\cite[Examples 2.2.2 and 2.2.3]{HL}). In the special case when
$S=\Spec k$, we get a moduli interpretation of the usual Grassmanian
of quotients.

The corresponding moduli interpretation of flag schemes does not seem
to be well-known so we sketch it below.

\medskip

Let $S$ be a scheme and let $\cE$ be a locally free $\cO_S$-module of
rank $(n+1)$.  Let us fix a sequence of integers $0<d_{1}<d_2<\cdots
<d_{m}<n+1$ for some $m\ge 1$.  A \emph{flag of type $(d_1,...,d_m)$
  in $\cE$} is a filtration
$$\cE_{1} \subset \cdots \subset \cE_{m}\subset \cE$$
such that  the sheaf $\cE_i$ is locally free of rank $d_i$ for every $i=1,...,m$ and all
quotients $\cE /\cE_{i}$ are locally free. If $f:T\to S$ is  a morphism of schemes  then we set $\cE_T:=f^*\cE$.

We define a functor $\Flagf (d_1,...,d_m; \cE)$
from the category of $S$-schemes to the category of sets by setting
$$(\Flagf (d_1,...,d_m; \cE))(T)= \{ \hbox{flags of type $(d_1,...,d_m)$  in }\cE_T  \}.$$
For a morphism of $S$-schemes $f:T_1\to T_2$ we define the
corresponding map $$(\Flagf (d_1,...,d_m; \cE))(T_2)\to (\Flagf
(d_1,...,d_m; \cE ))(T_1)$$ by pull-back.

\begin{theorem} \label{flag}
The functor $\Flagf (d_1,...,d_m; \cE)$ is representable and the corresponding $S$-scheme $\Flag (d_1,...,d_m; \cE)$
is projective.
\end{theorem}

\begin{proof}
  Let us first recall the Grassmannian but in the setting dual to the one
  described above. Let $\underline{\Gras} (r, \cE): (\Sch /S)^o\to
  \mathop{\rm Sets}$ denote the functor associating to an $S$-scheme
  $T$ the set of all locally free submodules $\cE'\subset \cE_T$ of
  rank $r$ with locally free quotient $\cE_T/\cE'$. This functor is
  represented by a projective $S$-scheme denoted by $\Gras _S(r,\cE)$
  (here we use the fact that $\cE$ is locally free of finite
  rank). This gives the required assertion for $m=1$.

  In general, let us consider the functor $\underline{\Gras} (d_1,
  \cE)\times _S... \times _S \underline{\Gras} (d_m, \cE)$ defined by
  associating to an $S$-scheme $T$ an $m$-tuple $\{\cE_i\subset \cE_T
  \}_{i=1,...,m}$ such that $\cE_i$ is locally free of rank $d_i$ and
  each quotient $\cE_T/\cE_i$ is locally free. The functor $\Flagf
  (d_1,...,d_m; \cE)$ is a subfunctor of the above functor defined by
  the conditions $\cE _i\subset \cE_{i+1}$ for $i=1,...,m-1$. Note
  that this subfunctor is closed, i.e., for any $S$-scheme $T$ and an
  $m$-tuple $\{\cE_i\subset \cE_T \}_{i=1,...,m}$ as above, there
  exists a closed subscheme $Z\subset T$ such that a morphism $f:
  T'\to T$ of $S$-schemes factors through $Z$ if and only if the
  $m$-tuple $\{f^*\cE_i\subset \cE_{T'} \}_{i=1,...,m}$ defines a
  filtration, i.e., we have $f^*\cE _i\subset f^*\cE_{i+1}$ for
  $i=1,...,m-1$.  The above $Z\subset T$ can be defined by vanishing
  of the canonical maps $\cE_i\to \cE_T\to \cE_T/\cE_{i+1}$ for
  $i=1,...,m-1$.  Since the functor $\underline{\Gras} (d_1,
  \cE)\times _S... \times _S \underline{\Gras} (d_m, \cE)$ is
  representable, the functor $\Flagf (d_1,...,d_m; \cE)$ is
  representable by a closed subscheme of the product of Grassmannians
  ${\Gras} (d_1, \cE)\times _S ... \times _S{\Gras} (d_m, \cE)$.
\end{proof}

One can also state and prove the corresponding theorem for an
arbitrary coherent $\cO_S$-module $\cE$ but then we need to use
equivalence classes of quotient filtrations as in the case of
Grassmannians.  More precisely, let $S$ be a scheme and let $\cE$ be a
coherent $\cO_S$-module.  Let us fix a sequence of integers
$0<d_{1}<d_2<\cdots <d_{m}$ for some $m\ge 1$.  A \emph{quotient flag
  of type $(d_1,...,d_m)$ in $\cE$} is a sequence of surjective maps
$$\cE\twoheadrightarrow \cE_{m}\twoheadrightarrow ... \twoheadrightarrow \cE_2\twoheadrightarrow \cE_{1} $$
such that  the sheaf $\cE_i$ is locally free of rank $d_i$ for every $i=1,...,m$.
We say that two quotient flags are equivalent if there exists a commutative diagram
$$ \xymatrix{
\cE\ar@{=}[d]\ar[r]&\cE_m\ar[d]\ar[r]&...\ar[r]& \cE_2\ar[r]\ar[d]&\cE_1\ar[d]\\
\cE\ar[r]&\cE_m'\ar[r]&...\ar[r]& \cE_2'\ar[r]&\cE_1'\\
}$$
in which all vertical maps are isomorphisms.
We define a functor $\Flagf (\cE; d_1,...,d_m)$
from the category of $S$-schemes to the category of sets by setting
$$(\Flagf (\cE; d_1,...,d_m))(T)= \{ \hbox{equivalence classes of quotient flags of type $(d_1,...,d_m)$  in }\cE_T  \}.$$
The functor on morphisms is again induced by pull-backs of quotient
flags.  Similar arguments as those used in proof of Theorem \ref{flag}
show that $\Flagf (\cE; d_1,...,d_m)$ is represented by a projective
$S$-scheme $\Flag (\cE; d_1,...,d_m)/S$.  If $\cE$ is locally free of
rank $(n+1)$ then the $S$-scheme $\Flag (\cE; d_1,...,d_m)/S$ is
isomorphic to the $S$-scheme $\Flag (n+1- d_m,...,n+1-d_1; \cE)/ S$.
The isomorphism is realised by sending a quotient flag
$\cE\twoheadrightarrow \cE_{m}\twoheadrightarrow
... \twoheadrightarrow \cE_2\twoheadrightarrow \cE_{1} $ to $\ker
(\cE\to \cE_1) \subset ...\subset \ker (\cE \to \cE_m)\subset \cE$.

\subsection{Moduli spaces of F-flags}

Note that in notation of Section \ref{DL-section} $X_{\GL (V)}= \Flag
(V;1,...,n)= \Flag (1,...,n;V^*)$ and the definition of flag schemes
does not require a choice of a flag.  However, Shubert varieties are
defined using choice of a Borel subgroup, i.e., a choice of the full
flag in $V$.  Nevertheless, it is possible to use the Frobenius
morphism to give a modular interpretation of the above closures of our
generalized Deligne--Lusztig $k$-schemes without any choice of a flag
in $V$. In fact, we give a more general construction that is useful
also in other contexts.

\medskip

Let $k=\FF_q$ and let $V$ be a $k$-vector space of dimension $(n+1)$.
Let us fix a sequence of integers $0<d_{1}<d_2<\cdots <d_{m}<n+1$ for
some $m\ge 1$.  Let us also fix a subset $S\subset \{d_1,...,d_m\}
\times \{d_1,...,d_m \} $. For $s\in S$ we denote by $s_1$ the first
coordinate and by $s_2$ the second one.  Note that if $T$ is a
$k$-scheme and $\cE_{1} \subset \cdots \subset \cE_{m}\subset V_T$ is
a flag of type $(d_1,...,d_m)$ then $\Fr ^* \cE_{1} \subset \cdots
\subset \Fr ^*\cE_{m}\subset \Fr ^*V_T=V_T$ is also a flag of type
$(d_1,...,d_m)$ in $V_T$.  We say that a flag $\cE_{1} \subset \cdots
\subset \cE_{m}\subset V_T$ of type $(d_1,...,d_m)$ is an \emph{F-flag
  of type $S$} if for all $s\in S$ we have $\Fr ^* \cE_{s_1}\subset
\cE_{s_2} $ if $s_1\le s_2$ and $ \cE_{s_2}\subset \Fr ^* \cE_{s_1} $
if $s_1> s_2$.

We define a functor $\Flagf ^F (S; V)$
from the category of $k$-schemes to the category of sets by setting
$$(\Flagf ^F(S; V))(T)= \{ \hbox{F-flags $\cE_{1} \subset \cdots \subset \cE_{m}\subset V_T$  of type } S  \}.$$
For a morphism of $S$-schemes $f:T_1\to T_2$ the corresponding
map $$(\Flagf ^F(S; V))(T_2)\to (\Flagf ^F (S; V ))(T_1)$$ is defined
by pull-back.

Similarly, one can define functors for quotient flags.

\medskip

Since the conditions defining the functor $\Flagf ^F (S; V)$ are
closed, in the same way as Theorem \ref{flag} one can prove the
following theorem:

\begin{theorem} \label{F-flag} 
  The functor $\Flagf ^F (S; V)$ is representable and the
  corresponding $k$-scheme is a projective subscheme of $\Flag
  (d_1,...,d_m; V)$.
\end{theorem}

If $S=\{ (s_1,s_2)\in   \{d_1,...,d_m\}
\times \{d_1,...,d_m \} : \, s_1<s_2\} $ then the scheme $\Flag ^F (S; V)$ is denoted by $\Flag ^F (d_1,...,d_m; V)$.
If $S=\{ (s_1,s_2)\in   \{d_1,...,d_m\}
\times \{d_1,...,d_m \} : \, s_1>s_2\} $ then the scheme $\Flag ^F (S; V)$ is denoted by $\Flag _F (d_1,...,d_m; V)$.

\begin{proposition}
We have $\Flag ^F (1,...,m; V)\simeq \tilde X_{m-1}$ for $m=1,...,n$.
\end{proposition}

\begin{proof}
  Restricting the universal family of flags on $X_{\GL (V)}$ to
  $\tilde X_{m-1}$ and using the universal property of $\Flag ^F
  (1,...,m; V)$ we get a canonical map $ \tilde X_{m-1}\to \Flag ^F
  (1,...,m; V)$. This map is compatible with embeddings into $X_{\GL
    (V)}$ and bijective on $\bar k$-points, so it is an isomorphism.
\end{proof}

\medskip

Note that there are many interesting, natural maps between moduli
spaces of F-flags. They are given by Frobenius pull-backs of some
factors in the flag.  For example if $m=2$  then we have the map
$$\Flag ^F (d_1,d_2; V)\to \Flag _F (d_1,d_2; V)$$
given on functors by sending flag $\cE_{1} \subset \cE_{2}\subset V_T$ to $\Fr^*\cE_{1} \subset \cE_{2}\subset V_T$
and the map
$$\Flag _F (d_1,d_2; V)\to \Flag ^F (d_1,d_2; V)$$
given by sending flag $\cE_{1} \subset \cE_{2}\subset V_T$ to $\cE_{1}
\subset \Fr^* \cE_{2}\subset V_T$.  The composition of these maps
sends $\cE_{1} \subset \cE_{2}\subset V_T$ to $\Fr^*\cE_{1} \subset
\Fr^* \cE_{2}\subset V_T$, so it is equal to the Frobenius
endomorphism of ${\Flag ^F (d_1,d_2; V)}$.

Let us take $n=2$, $d_1=1$ and $d_2=2$.  Then the above maps
correspond to $\varphi_{X_V}$ and $\varphi _{X_{V^*}}$.  We have a
canonical isomorphism $ \Flag ^F (1,2; V)\simeq \Flag _F (1,2; V^*) $
given by sending flag $\cE_{1} \subset \cE_{2}\subset V_T$ to
$\cE_1'=\ker (V_T^* \to \cE_{2}^*) \subset \cE_2'=\ker (V_T^* \to
\cE_{1}^*) \subset V^*_T.$ Therefore $ \Flag ^F (1,2; V)$ is
non-canonically isomorphic to $\Flag _F (1,2; V)$. If we choose an
isomorphism $V\simeq V^*$ then the above maps become the same
endomorphism of $\Flag ^F (1,2; V)$. One can also obtain a similar
modular interpretation of the maps from Theorem \ref{main-3'} in
higher dimensions.

\section{Logarithmic tangent bundle}\label{section-logarithmic}

Let $X$ be a smooth projective variety with a divisor $D$, both defined over some field $k$.
Let $j:U\hookrightarrow X$ be the open subset where the pair $(X, D)$ is smooth. Then we can consider
the logarithmic tangent bundle $T_U(-\log D|_U)$. Since $U$ is a big open subset
of $X$ (i.e., the codimension of the complement is at least $2$), the sheaf $T_X(-\log D):=j_* (T_U(-\log D|_U))$
is reflexive and it is called \emph{the logarithmic tangent sheaf} of $(X,D)$.
We will also use a well-known interpretation of $T_X(-\log D)$ as the sheaf of derivations preserving the ideal
sheaf of $D$.

We say that $D$ is \emph{free} if $T_X(-\log D)$ is locally free.

\medskip

As in Subsection \ref{Hirokado} we consider the maps $\theta _i: \cO_{\PP
  (V)} (-q^i+1)\to T_{\PP (V)}$ defined by $(\Fr ^i)^*s\otimes \id
_{\cO _{\PP (V)}(1)}$ for $i=1,...,n$.  Let us set $$\theta=(\theta
_1,...,\theta_{n}): \bigoplus _{i=1}^{n}\cO_{\PP (V)}
(-q^i+1){\to} T_{\PP (V)} .$$
By Lemma \ref{Moore} this map is an isomorphism outside of the sum $B$ of all $k$-linear hyperplanes
in $\PP (V)$. In particlar, $\theta$ is injective as a map of $\cO_{\PP(V)}$-modules.

\begin{proposition} \label{log-tangent}
$\theta$ induces an isomorphism $T_{\PP (V)}(-\log \,  B)\simeq  \bigoplus _{i=1}^{n}\cO_{\PP (V)}
(-q^i+1)$. In particular, $B$ is a free hyperplane arrangement on $\PP (V)$.
\end{proposition}

\begin{proof}
        Let $\bar B$ be the sum of all $k$-linear hyperplanes in vector space $V^*$
        and let $\nu: V^*_0:=V^*-\{0\}\to \PP (V)$ denote the canonical projection. Let us fix linear coordinates in $V^*$ and let us consider derivations $\delta_i:=\sum_{j=0}^n x_j^{q^i}\frac{\partial}{\partial x_j}$ of $\cO_{V^*}$. If $a_j\in \FF_q$ then
    $$\delta_i (\sum a_jx_j)=(\sum a_jx_j)^{q^i}\in (\sum a_jx_j) \cO_{V^*},$$
    so derivation $\delta _i$ preserves the ideal of $\bar B$. Hence the image of
    $(\delta_0,...,\delta_n): \cO_{V^*}^{n+1}\to T_{V^*}$ lies in $T_{V^*}(-\log \bar B)$. Then comparison of the first Chern classes shows that the image of this map coincides with
    $T_{V^*}(-\log \bar B)$ (this fact is usually called Saito's criterion; see also \cite[Example 4.24]{OT}).   Euler's exact sequence implies that
    we have an exact sequence
    $$0\to \cO_{V^*_0}\mathop{\longrightarrow}^{\delta _0}T_{V^*}(-\log \bar B)|_{V^*_0}
    \mathop{\longrightarrow}^{\eta} \nu^*T_{\PP(V)}(-\log B) .$$
    Since $\nu$ is a $\GG_m$-torsor and $T_{\PP(V)}(-\log B)$ is reflexive,
    $\nu^*T_{\PP(V)}(-\log B)$ is also reflexive. On the other hand, $\delta_0$ defines a nowhere vanishing section of  $T_{V^*}(-\log \bar B)|_{V^*_0}$, so the cokernel of the corresponding map is locally free. It follows that  $\eta$ induces an injection of a locally free sheaf into a reflexive sheaf. Since $\eta$ is generically surjective, it
    is surjective and $\nu^*T_{\PP(V)}(-\log B)$ is locally free. Hence using descent we see that $T_{\PP(V)}(-\log B)$ is also locally free.
    Since the maps $\eta\delta_i: \cO_{V^*_0}\to \nu^*T_{\PP(V)}(-\log B)$ for $i=1,...,n$ descend to
    $\theta _i: \cO_{\PP (V)} (-q^i+1){\to} T_{\PP (V)}(-\log B)$, we get the required assertion.
    \end{proof}

\section{1-forms and foliations on the wonderful compactification}\label{section-foliation}

For $j=1,...,n$ let us denote by  $\cF_{j}\subset T_{\PP (V)}$ the image of $\bigoplus
_{i=1}^{j}\cO_{\PP (V)} (-q^i+1)$ under $(\theta
_1,...,\theta_{j})$.
By  Corollary \ref{Eagon-Northcott} the map $V\otimes \cO_{\PP (V)}\to \bigoplus
_{i=0}^{j}\cO_{\PP (V)} (q^i) $ is surjective outside of $Z_{n+1-j}$. By Euler's exact sequence this map
induces a generically surjective map
$$\Omega_{\PP(V)}\to \bigoplus _{i=1}^{j}\cO_{\PP (V)} (q^i-1).$$
Dualizing this map we see that for $j=1,...,n-1$ the quotient $T_{\PP (V)}/\cF_{j}$ is torsion free and locally free
outside of $Z_{n+1-j}$.

An easy computation shows that
$$[\theta_i, \theta_j]=\theta_j-\theta _i,$$
so $[\cF_j, \cF_j]\subset \cF_j$. Using this equality, Lemma \ref{lemma-Hirokado} and
Jacobson's identity one can also check that $\cF_j^p\subset \cF_j$,
so $\cF_j\subset T_{\PP (V)}$ is a $1$-foliation for $j=1,...,n-1$.
By Proposition \ref{log-tangent} we have an increasing filtration
$$\cF _1\subset \cF_2\subset ...\subset \cF_n=T_{\PP (V)}(-\log B)\subset T_{\PP(V)}.$$
Note that $T_{\PP (V)}(-\log B)/\cF_j$ is locally free, so $\cF_j\subset T_{\PP (V)}(-\log B)$ are ``smooth  logarithmic $1$-foliations'' (note however that the pair $(\PP (V), B)$ is not log smooth).
One can also show that singularities of $\cF_j\subset T_{\PP (V)}$ are resolved by passing to $\tilde X_j$, i.e., the foliation induced by $\cF_j\subset T_{\PP (V)}$ on $T_{\tilde X_j}$ is smooth. 
Hirokado's construction from \cite{Hi} uses this fact in the special case when $n=3$ and $j=1$.
Since we do not need this fact in general, we skip its proof.

\medskip

Let us consider the maximal $1$-foliation $\cF_{n-1}\subset T_{\PP (V)}$. We have an exact sequence
$$0\to \cF_{n-1} \to T_{\PP (V)}\to I_Z((1+q+...+q^{n-1})+1)\to 0.$$
Let $\omega $ be a rational $1$-form  corresponding to the last non-zero map in this sequence.
It fits into an exact sequence
$$0\to \cO_{\PP (V)}(-1-q-...-q^{n-1})\mathop{\longrightarrow}^{\omega} \Omega_{\PP (V)}(1)
\mathop{\longrightarrow}^{\theta(-1)^{\vee}} \bigoplus _{i=1}^{n-1}\cO_{\PP (V)}
(q^i).$$
Let us fix a basis $e_0,...,e_n$ of $V$. Then we have
$$\omega=\sum_{i=0}^n(-1)^i\Delta_q(x_0,...,\hat x_i,...,x_n) e_i,$$
when treated as an element of
$$H^0(  \cO_{\PP (V)}(1+q+...+q^{n-1})\otimes \Omega_{\PP (V)}(1))\subset H^0(\cO_{\PP (V)}(1+q+...+q^{n-1})\otimes V).$$

\medskip

\begin{proposition} \label{codim-1-foliation}
The rational $1$-form $\omega$ induces a $1$-foliation on the wonderful compactification $\tilde X$ of Drinfeld's half-space $\Omega (V)$. More precisely, we have a short exact sequence
$$0\to \tilde \cF_{n-1} \to T_{\tilde X}\to \cO_{\tilde X} \left( \left(\frac{q^{n}+q-2}{q-1}\right)H
-\sum _{i=1}^{n-1} \frac{q^{i}+q-2}{q-1}D_{V}^{i+1}\right)\to 0$$
in which $\tilde \cF_{n-1}$ is a smooth $1$-foliation of rank $(n-1)$ induced by $\cF_{n-1}\subset T_{\PP (V)}$.
\end{proposition}

\begin{proof}
Let $U=\Spec k[t_1,...,t_n]\subset \PP (V)$ correspond to $D_{+}(x_0)$ with $t_i=x_i/x_0$.
Then
$$\omega|_U=\sum_{i=1}^n(-1)^i\Delta_q(1,t_1,...,\hat t_i,...,t_n)\, dt_i.$$
Let us note that
\begin{eqnarray*}
h_n&:=&\prod_{i=1}^n\prod_{(a_i,...,a_n)\in k^{n+1-i}}(1+a_is_i+a_{i+1}s_is_{i+1}+...+a_ns_i...s_n)=\prod _{i=1}^n\frac {\Delta_q(1,s_i,s_is_{i+1},...,s_i...s_n)}{\Delta_q(s_i,s_is_{i+1},...,s_i...s_n)}\\
&=& \frac{\Delta_q(1,s_1,s_1s_{2},...,s_1...s_n)}{\Delta_q(s_n)}\cdot \prod _{i=1}^{n-1}
\frac {\Delta_q(1,s_{i+1},s_{i+1}s_{i+2},...,s_{i+1}...s_n)}{\Delta_q(s_{i},s_is_{i+1},...,s_i...s_n)}
=\frac {\Delta_q(1,s_1,s_1s_{2},...,s_1...s_n)}{\prod _{i=1}^{n} s_i^{q^0+q^1+...+q^{n-i}}}.
\end{eqnarray*}
The map $\pi: \tilde X \to \PP (V)$ is covered by charts that are given 
$$\pi: \tilde U:= D(h_n)\subset \Spec k[s_1,...,s_n]\to \Spec k[t_1,...,t_n]$$
with $\pi^*(t_i)=\prod_{j\le i}s_j$ (see \cite[proof of Lemma 1.3]{RTW} but note that the last diagram on [ibid., p. 1214] needs to be restricted  to an open subset of $Y$ as the map $q$ is given by a different formula on whole $\Spec k[t_1,...,t_n]$;  there is also a misprint in description of the exceptional divisor in [ibid.]).
The exceptional divisor of this map is given by $\mathrm{div}\, (s_1...s_{n-1})$ and we have
$$\pi^*\left(\frac{dt_i}{t_i}\right)=\sum_{j\le i} \frac{ds_j}{s_j}.$$
Therefore we obtain
\begin{eqnarray*}
\pi^*\omega|_{\tilde U}&=&\sum_{i=1}^n(-1)^i\Delta_q(1,s_1,...,\widehat{s_1...s_i},...,s_1...s_n)
\left( s_1...s_i \sum_{j\le i} \frac{ds_j}{s_j}\right)=\\
&=&\sum_{j=1}^n \left(
\sum_{i\ge j} (-1)^is_1...\hat s_j...s_i\cdot \Delta_q(1,s_1,...,\widehat{s_1...s_i},...,s_1...s_n)
  \right) ds_j .
\end{eqnarray*}
The coefficient of $\pi^*\omega|_{\tilde U}$ at $ds_n$ is equal to
$$(-1)^{n}s_1...s_{n-1}\Delta_q(1,s_1,...,s_1...s_{n-1})
 =(-1)^n
h_{n-1}\cdot  \prod _{i=1}^{n-1} s_i^{(q^0+q^1+...+q^{(n-1)-i})+1}.$$
It is  easy to see that
the coefficients of $\pi^*\omega|_{\tilde U}$ at $ds_j$ for $j=1,...,n-1$ are
also divisible by $\prod _{i=1}^{n-1} s_i^{(q^0+q^1+...+q^{(n-1)-i})+1}$.
Since the polynomial $h_{n-1}$ divides $h_n$, we can write
$$\pi^*\omega|_{\tilde U}=\prod _{i=1}^{n-1} s_i^{(q^0+q^1+...+q^{(n-1)-i})+1}
(\alpha_1 ds_1+...+\alpha_n ds_n)$$
for some $\alpha_i\in k[s_1,...,s_n]$ such that $\alpha_n$ is invertible in $ \cO(D(h_n))=
k[s_1,...,s_n]_{h_n}$.  This shows that $\pi^*\omega$
defines the map
$$ \cO_{\tilde X} \left( - \left(\frac{q^{n}+q-2}{q-1}\right)H
+\sum _{i=1}^{n-1} \frac{q^{i}+q-2}{q-1}D_{V}^{i+1}\right) \to \Omega_{\tilde X}$$
with locally free cokernel. After dualizing we get the required smooth $1$-foliation on the wonderful compactification of
$\Omega (V)$.
\end{proof}

\medskip

A simple computation of the canonical divisor of a blow up shows that
$$K_{\tilde X}=-(n+1)H+\sum _{i=1}^{n-1} i\, D_{V}^{i+1}.$$
Therefore we have
\begin{eqnarray*}
c_1(\det \tilde \cF_{n-1})&=& -K_{\tilde X}-\left(
\left(\frac{q^{n}-1}{q-1} +1\right)H
-\sum _{i=1}^{n-1} \left( \frac{q^{i}-1}{q-1} +1\right)
D_{V}^{i+1}\right) \\
&=& -\left(\frac{q^{n}-1}{q-1} -n\right) H+\sum _{i=1}^{n-1}
\left(\frac{q^{i}-1}{q-1} -i+1\right) D_{V}^{i+1} .
\end{eqnarray*}

\medskip

Let us note that
$$\pi ^*\left(t_i\frac{\partial}{\partial t_i}\right)=\left\{ \begin{array}{cl}
s_i\frac{\partial}{\partial s_i}-s_{i+1}\frac{\partial}{\partial s_{i+1}}&
\hbox{for $i<n$,}\\
s_n\frac{\partial}{\partial s_n}&
\hbox{for $i=n$.}\\
\end{array}\right.
$$
Let $\delta_j$ be the rational vector field corresponding to $\theta_j$. Then in the notation of proof
of Proposition \ref{codim-1-foliation} we can write
$$\delta_j|_U=\sum _{i=1}^n(t_i^{q^j-1}-1)t_i\frac{\partial}{\partial t_i} .$$
Therefore after a short computation we get
$$\pi^*\delta_j|_{\tilde U}=(s_1^{q^j-1}-1) s_1\frac{\partial}{\partial s_1} +s_1^{q^j-1}(s_2^{q^j-1}-1)s_2\frac{\partial}{\partial s_2}+...+(s_1...s_{n-1})^{q^j-1}(s_n^{q^j-1}-1)s_n\frac{\partial}{\partial s_n}.$$

Let us set $D:=\sum _{i=1}^nD^i_V$.

\begin{proposition} \label{dim-1-foliation}
The rational  vector field  $\delta_1$ induces a smooth $1$-foliation
$$\tilde \cF_1\simeq   \cO_{\tilde X} \left(  -(q-1)H +D_{V}^{n}\right) \subset T_{\tilde X}$$ on the wonderful compactification $\tilde X$ of Drinfeld's half-space $\Omega (V)$.\end{proposition}

\begin{proof}
Taking $j=1$ in the above formula on $\pi^*\delta_j|_{\tilde U}$, we see that $\pi^*\delta_1|_{\tilde U}$ vanishes with order one along $\mathrm{div}\, (s_1)$ and it does not vanish along any other component of the exceptional divisor
 $\mathrm{div}\, (s_1...s_{n-1})$.  So the corresponding foliation $\tilde \cF_1$ is isomorphic to
 $\cO_{\tilde X} \left(  -(q-1)H +D_{V}^{n}\right)$.  Since $\cF_1\subset T_{\PP (V)}$
has singularities only along $Z_n=\PP(V)(\FF_q)$, this foliation can have singular points only 
on $\mathrm{div}\, (s_1)$. Since $s_1^{q-1}-1$ does not vanish on this divisor,  $\cF_1$ is smooth.
\end{proof}

The above proposition implies that although $\cF_j\subset T_{\PP (V)}(-\log B)$ for all $j$, already
$\tilde \cF_1$ is not contained in  $T_{\tilde X}(-\log\, D)$. One can generalize the above proposition and describe the filtration
$$\tilde \cF _1\subset\tilde \cF_2\subset ...\subset \tilde \cF_n=T_{\tilde X}$$
and the induced filtration on $T_{\tilde X}(-\log\, D)$. More precisely, each $\tilde \cF_j$ is a smooth foliation of rank $i$ and 
$$\tilde \cF_j/\tilde \cF_{j-1}\simeq \cO_{\tilde X}\left( -(q-1)H_{n+2-j}+D_V^{n+1-j}\right)$$
(see Proposition \ref{seq-fol-dim3} for the $3$-dimensional case).

Let $\cL_j$ be the image of $\theta_j$ in $T_{\PP (V)}$ and let $\tilde \cL_j\subset T_{\tilde X}$ be the corresponding foliation. Let us also set $\cM_j:=\tilde \cL_j\cap T_{\tilde X}(-\log\, D)$. 
Note that for $j\ge 2$ the foliation $\tilde \cL_j$ is not smooth as $\pi: \tilde X\to \PP(V)$ is an isomorphism over $\PP(V)-Z$ and for $j\ge 2$ the foliation $\cL_j\subset T_{\PP (V)}$ has singularities along a non-empty set
of $\FF_{q^j}$-points of $\PP(V)-Z$.
However, $\cM_j \simeq \cO_{\tilde X} \left(  -(q^j-1)H \right)$
and $T_{\tilde X}(-\log\, D)/\cM_j$ is locally free. In particular, we have
$$\bigoplus _{i=1}^n \cM_i\simeq \bigoplus _{i=1}^n\cO_{\tilde X} \left(  -(q^i-1)H \right)\subset T_{\tilde X}(-\log\, D).$$
Let us set $\cG_0=0$ and $\cG_j=\tilde \cF_j\cap T_{\tilde X}(-\log\, D)$ for $j=1,...,n$. By definition $\cG_j$ is the saturation of $\bigoplus _{i=1}^j \cM_i$ in $T_{\tilde X}(-\log\, D).$ Then the quotients of the filtration
$\cG_{\bullet}$ of $T_{\tilde X}(-\log\, D)$ are  given by
$$\cG_{j}/\cG_{j-1}\simeq \cO_{\tilde X}(-(q-1)H_{n+2-j})=\cO_{\tilde X}\left( -(q^j-1)H+ \sum _{i=0}^{j-2} \left(q^{i+j-1}-1\right)D^{n-i}_V\right) $$ 
for $j=1,...,n$. We leave the proofs of these facts to an interested reader (for the last fact see also Proposition \ref{foliations} in the surface case).

\section{Rational surfaces whose cotangent bundle contains an ample line bundle}

Let $\pi:X\to \PP^2$ be the blow up of $\PP^2$ along $Z=\PP^2(\FF_q)$. We denote by
$H$ the pull back of the hyperplane divisor on $\PP ^2$ and by $E=\sum E_i$  the exceptional divisor of $\pi$.
Let $\tilde L_i$ for $i=1,...,q^2+q+1$ be the strict transforms of lines $L_i$ passing through at least two points
of $Z$. Let us set $B=\sum L_i$, $\tilde B=\sum \tilde L_i$ and $D=\tilde B+E$.

The following proposition should be compared with its non-explicit version \cite[Lemma 8.3]{La}:

\begin{proposition} \label{foliations}
We have short exact sequences
$$0\to \cO _{ X}(-(q+2)H+2E)\to \Omega_{ X}\to \cO _{X}((q-1)H-E)\to 0,$$
$$0\to \cO _{ X}((q^2-1)H-(q-1)E)\to \Omega_{X}(\log \tilde B)\to \cO _{ X}((q-1)H-E)\to 0$$
and 
$$0\to \cO _{ X}((q^2-1)H-(q-1)E)\to \Omega_{ X}(\log D)\to \cO _{ X}((q-1)H)\to 0.$$
In particular, $\Omega_{\tilde X}(\log \tilde B)$ and $\Omega_{\tilde X}(\log D)$ contain $M=(q^2-1)H-(q-1)E$ such that $M^2>0$ and $MH>0$.
\end{proposition}

\begin{proof}
The first sequence  comes from dualizing the inclusion $\tilde \cF_1\subset T_{X}$  described in Proposition \ref{codim-1-foliation}. In the notation of proof of this proposition we have
$$\pi^*\delta_1 |_{\tilde U}=s_1\left[(s_1^{q-1}-1)\frac{\partial}{\partial s_1}+s_1^{q-2}(s_2^{q-1}-1)s_2\frac{\partial}{\partial s_2}\right] .$$
Note that $\delta_1$ defines $\cO _{\PP^2}(1-q)\to T_{\PP^2}(-\log B)$ and the cokernel is locally free
(see Proposition \ref{log-tangent}). Since  $\pi^*\delta_1$ vanishes with order $1$ along the exceptional divisor $E$, it
defines a map $\cO _{X}(-(q-1)H+E)\to T_{X}(-\log \tilde B)$, which by the above is locally free outside of $E$. It is also locally free outside of $\tilde B$, so we need only to check that the cokernel is locally free at all points of 
intersection of $\tilde B$ and $E$.  On $\tilde U$ the intersection $\tilde B\cap \pi^{-1}(0)$ is just $1$ point given by $s_1=s_2=0$. Since at this point $T_{ X}(-\log \tilde B)$ has local generators $\frac{\partial}{\partial s_1}$ and $s_2\frac{\partial}{\partial s_2}$,  the cokernel of  $\cO _{ X}(-(q-1)H+E)\to T_{ X}(-\log \tilde B)$  is locally free
also at this point.  After dualizing  this shows existence of the second exact sequence. 

To see the last exact sequence note that  by similar arguments $\pi^*\delta_1$ defines  also the map $\cO _{ X}(-(q-1)H)\to T_{ X}(-\log D)$ with locally free cokernel.
\end{proof}

It is interesting to note that the last sequence from the above proposition is usually  (but not always) split. 
It will be more convenient to consider the dual sequence.

\begin{proposition} \label{splitting-sequence}
The short exact sequence
\begin{align}
0\to \cO _{ X}(-(q-1)H)\to T_{X}(-\log D)\to \cO _{ X}(    -(q^2-1)H+(q-1)E)\to 0 \label{seq}
\end{align}
splits (canonically) if and only if $q\ne 2$.
\end{proposition}

\begin{proof}
First let us remark that if the sequence (\ref{seq}) splits then the splitting $ 
 \cO _{ X}(    -(q^2-1)H+(q-1)E) \to T_{ X}(-\log D)$ is unique (up to multiplication by a non-zero constant). To prove this one needs to remark
that $\Hom (  \cO _{X}(    -(q^2-1)H+(q-1)E) ,  \cO _{X}(- (q-1)H))=H^0( \cO _{ X}((q^2-q)H-(q-1)E))=0$.
The proof is the same as that of \cite[Lemma 2.6]{Hi}.

Note that  after push-forward by $\pi$  the above sequence becomes the sequence
$$0\to \cO _{X}(-(q-1))\to T_{\PP ^2}(-\log B)\to \cO _{\PP^2}(    -(q^2-1))\to 0,$$
which always splits by Proposition \ref{log-tangent}. 
Now the splitting $\cO _{\PP^2}(    -(q^2-1))\to T_{\PP ^2}(-\log B)$ of this sequence can be written (up to multiplication
by a constant) as $\eta=\theta_2+F \theta_1$ for some homogeneous polynomial $F\in k[x_0,x_1,x_2]$ of degree $q^2-q$. The sequence (\ref{seq}) splits if and only if for some $F$ the (rational) vector field $\pi^*\eta$ defines  splitting. As before let us consider  subset $U=D_{+}(x_0)\subset \PP ^2$  with affine coordinates $t_1=x_1/x_0$ and $t_2=x_2/x_0$.
Let us take an $\FF_q$-point $P=(a_1,a_2)\in U$. Let $u_1, u_2$ be the affine coordinates on $U$ given by $u_i=t_i-a_i$ for $i=1,2$. Then we can write
$$\eta|_U= \sum_{i=1}^2\left((u_i^{q^2}-1) - F(1, a_1+u_1, a_2+u_2)(u_i^{q-1}-1)\right) u_i\frac{\partial}{\partial u_i}. 
$$
Let $V\subset U$ be an open subset containing only one $\FF_q$-rational point $P$ and let $\tilde V=\pi^{-1} (V)$.
On an affine chart of $\tilde V$ with coordinates $s_1,s_2$ such that $\pi^*u_1=s_1$ and $\pi^*u_2=s_1s_2$ we can write
\begin{eqnarray*}
\pi^*\eta|_{\tilde V}=&\left((s_1^{q^2}-1) - F(1, a_1+s_1, a_2+s_1s_2)(s_1^{q-1}-1)\right) s_1\frac{\partial}{\partial s_1}
\\
&+
 s_1^{q-1}\left( s_1^{q^2-q}(s_2^{q^2-1}-1) -F(1, a_1+s_1, a_2+s_1s_2) (s_2^{q-1}-1)  \right)s_2\frac{\partial}{\partial s_2}
\end{eqnarray*} 
$\pi ^* \eta$ defines the map $ \cO _{ X}(    -(q^2-1)H) \to T_{X}(-\log D)$
which defines splitting of  the sequence (\ref{seq}) on $\tilde V$ if and only if   $\left((s_1^{q^2}-1) - F(1, a_1+s_1, a_2+s_1s_2)(s_1^{q-1}-1)\right)$ vanishes up to order $(q-1)$ along the exceptional divisor given by $s_1=0$ (for all $P\in \PP^2(\FF_q)$).
This happens (on $\tilde V$) if and only if $F(1,a_1,a_2)=1$. It follows that  the sequence (\ref{seq}) splits if and only if there exists 
a homogeneous polynomial $F\in k[x_0,x_1,x_2]$ of degree $q^2-q$ such that
$$F(a_0,a_1,a_2)=1$$
for all $(a_0,a_1,a_2)\in \FF_q^3-\{0\}$ (note that  if $a_0\ne 0$ then $F(a_0,a_1,a_2)=a_0^{q^2-q}F(1,a_1/a_0,a_2/a_0)=1$; similarly we have equalities for other charts).
If $q=2$ then such a polynomial does not exists as every conic defined over $\FF_q$ has an $\FF_q$-rational point 
(this follows from Chevalley--Warning theorem). 
Now we claim that if $q\ne 2$ then there exists a degree $q$ plane projective curve defined over $\FF_q$ with no 
$\FF_q$-rational points. This is equivalent to finding a homogeneous polynomial $G\in k[x_0,x_1,x_2]$ of degree $q$ 
that has no zeroes in $k^3-\{0\}$. 
To find it let us choose an extension $k\subset k'$ of degree $q$. Then the norm map $N_{k'/k}: k'\to k$ is given by a homogenous polynomial of degree $q$ in $q$ variables. Restricting it to a $3$-dimensional $k$-vector subspace of $k'$
and choosing some coordinates we get the required polynomial $G$.
Now the required homogeneous polynomial can be obtained  by setting $F=G^{q-1}$. 
\end{proof}

\medskip

\begin{remark}
Note that if $q>2$ then there exists only one polynomial  $F\in k[x_0,x_1,x_2]$ of degree $q^2-q$ such that
$$F(a_0,a_1,a_2)=1$$
for all $(a_0,a_1,a_2)\in \FF_q^3-\{0\}$.
This is non-obvious from the construction of this polynomial  
and it follows from uniqueness (up to multiplication by a constant) of splitting of the sequence (\ref{seq}).
\end{remark}

\medskip

Note that $\tilde B$ is a sum of $(q^2+q+1)$ curves $\tilde L_i\simeq \PP^1$
with self-intersection $\tilde L_i^2=-q$.  Let $f:{ X}\to Y$ be the
contraction of $\tilde L_1,..., \tilde L_{q^2+q+1}$ to a set of points
$Q_1,...,Q_{q^2+q+1}$.  Let us also set $E_Y=f(E)$. Since $E_Y$ is an
effective $\QQ$-Cartier Weil divisor and $\rho (Y)=1$, it is also
ample.  Clearly, $Y$ has $\QQ$-factorial klt singularities at $Q_i$
of the type considered in Section \ref{cone}. We show that $\hat \Omega_Y$
contains an ample line bundle. Since $\Omega _X$ does not contain big
line bundles, we see that $f_*\Omega_X\ne \hat \Omega_Y$ as predicted
by Proposition \ref{main-2'}.

\begin{proposition} \label{B_2} The sheaf $f_*\Omega_{ X}(\log \tilde B)$ is reflexive and we
have a short exact sequence
$$0\to \cO _Y(B_1)\to \hat \Omega_Y\to \cO _Y(B_2)\to 0,$$
where $B_1$ is an ample Cartier divisor and $B_2$ is an anti-ample $\QQ$-Cartier Weil divisor.
\end{proposition}

\begin{proof}
Since we have
$$((q^2-1)H-(q-1)E)\tilde L_i=0,$$
locally around $\tilde L_i$, the sheaf $\cO_{{ X}}((q^2-1)H-(q-1)E)$ is isomorphic to $\cO_{ X}$. In particular,
$f_*\cO _{ X}((q^2-1)H-(q-1)E)$ is locally free and  $R^1f_*\cO _{ X}((q^2-1)H-(q-1)E)=0$. Since
$$((q-1)H-E)\tilde L_i=-2,$$
Lemma \ref{Wahl's reflexivity} shows that $f_* \cO _{ X}((q-1)H-E)$ is reflexive.
Let us recall that $(q^2+q+1)H\sim \tilde B+(q+1)E$. Hence we have
$$(q^2-1)H-(q-1)E\sim _{\QQ} \frac{q^2-q}{q^2+q+1}E+\frac{q^2-1}{q^2+q+1}{\tilde B}$$
and
$$(q-1)H-E\sim _{\QQ} -\frac{q+2}{q^2+q+1}E+\frac{q-1}{q^2+q+1}{\tilde B}.$$
Therefore
$B_1=f_*((q^2-1)H-(q-1)E)\sim_{\QQ}\frac{q^2-q}{q^2+q+1}E_Y$
is ample and
$B_2=f_*((q-1)H-E)\sim_{\QQ}-\frac{q+2}{q^2+q+1}E_Y$
is anti-ample.  Now the required short exact sequence can be obtained  by pushing down the second exact sequence from Proposition \ref{foliations}.
\end{proof}

The above proposition completes the proof of Theorem \ref{B-S-failure}.

\bigskip

\begin{example}
  
  Now we can use the $1$-foliations from Section
  \ref{section-foliation} to give a simple interpretation of Theorem
  \ref{main-3} in the case $n=2$ and $k=\FF_p$. Let us recall that the
  map $\pi : \tilde X=X\to \PP^2$ is the blow up along $\PP^2(\FF_p)$
  and the map $f$ is the contraction of all the strict transforms of
  the $\FF_p$-lines on $\PP^2$. The remaining maps can be identified
  as follows.  The map $\varphi: \PP^2\to Y$ is the quotient by
  $1$-foliation defined by $\cF_1\subset T_{\PP ^2}$, the map
  $\varphi_X: X\to X$ is the quotient by the $1$-foliation $\tilde
  \cF_1\subset T_{X}$ and the map $\varphi_Y: Y\to \PP^2$ is the
  quotient by the $1$-foliation $\cO_Y(-B_2)\subset T_Y$ coming by
  dualization from Proposition \ref{B_2}.  The necessary computations
  are left to the reader or they can be found in the preprint version
  of the paper.
\end{example}

\medskip

Note that if $q=p$ then $(p+1)H-E$ and $H$ are both globally generated
and they define maps contracting different curves. Therefore
$a((p+1)H-E)+bH$ is ample for any positive real numbers $a$ and $b$.
In particular $A=(p+\sqrt{p}+1)H-E$ is ample. One can easily check
that $\varphi_X^*A=\sqrt{p}\, A$, so $\varphi_X:X\to X$ is a polarized
endomorphism. However, this polarization is an $\RR$-divisor and it
cannot be chosen to be a $\QQ$-divisor.

\medskip

In fact, we have the following more general lemma suggested by De-Qi Zhang:

\begin{lemma}\label{De-Qi}
Let $\varphi: X\to X$ be an endomorphism of a scheme $X$ and let us assume that for some $n\ge 1$, $\lambda \in \RR$ and some line bundle $L$ we have $(\varphi ^n)^*L\equiv \lambda ^n L$. If we set   $M=(L+\frac{1}{\lambda}\varphi ^*L+...+\frac{1}{\lambda^{n-1}}(\varphi^{n-1}) ^*L)$ then
we have $\varphi^* M\equiv{\lambda} M$. In particular, if $\varphi^n$ is polarized then $\varphi$ is also polarized.
\end{lemma}

The proof is an easy computation left to the reader.

\medskip

Let $\varphi: X\to X$ be an endomorphism and let $H$ be an ample divisor such that $\varphi ^*H\equiv \lambda H$. Let us set $X^n:= \underbrace{X\times ...\times X}_{n}$.
Then we can consider the map $\tilde \varphi : X^n \to X^n$
given by
$$\tilde \varphi (x_1,...,x_n)=(\varphi (x_n), x_1,...,x_{n-1}).$$
Let us set $A=(H,...,H)$. Then we have $\tilde{\varphi}^n
=\underbrace{\varphi \times ...\times \varphi}_{n}$ and
$(\tilde{\varphi}^n)^*A\equiv \lambda A$.  Therefore
$(\tilde{\varphi}^n)^*A\equiv \lambda A$ and Lemma \ref{De-Qi} implies
that there exists an ample $\RR$-divisor $\tilde A$ such that
$\tilde{\varphi}^*\tilde A\equiv \sqrt[n]{\lambda}\cdot \tilde A$.
Clearly, the same equality holds if we restrict $\tilde A$ to any
subvariety $Y\subset X^n$ preserved by $\tilde{\varphi}$. However, it
is usually difficult to find such a subvariety.

The morphism $\varphi_X:X\to X$ from Theorem \ref{main-3} is obtained
by this construction applied to the Frobenius morphism $\Fr _{\PP
  ^n}:\PP^n\to \PP^n$.

\section{A new non-liftable Calabi--Yau threefold in characteristic $2$}

Let us consider the $3$-dimensional case, i.e. $n=3$, and let $\tilde f: \tilde X\to \tilde Y=\tilde X/\tilde \cF_{2}$ be the quotient by $1$-foliation.  The main aim of this section is to prove that for $p=q=2$ the obtained variety $\tilde Y$
is a new example of a smooth projective Calabi--Yau $3$-fold that does not lift to characteristic $0$ (see Theorem \ref{new-CY'}). To study this quotient we need some preparatory lemmas.

\begin{lemma} \label{seq-fol-dim3}
We have a short exact sequence
$$0\to \cO_{\tilde X} \left(  -(q-1)H +D_{V}^{3}\right)  \to \tilde \cF _2\to \cO_{\tilde X} \left(  -(q^2-1)H +(q-1)D_V^3+ D_{V}^{2}\right) \to 0.$$
\end{lemma}

\begin{proof}
By  Proposition \ref{codim-1-foliation} and Proposition \ref{dim-1-foliation} we know that $T_{\tilde X}/\tilde \cF_i$ are locally free for $i=1,2$. Therefore the sequence
$$0\to \tilde \cF_2/\tilde \cF_1 \to T_{\tilde X}/\tilde \cF_1\to T_{\tilde X}/\tilde \cF_2\to 0$$
shows that $\tilde \cF_2/\tilde \cF_1$ is a line bundle.
Let us recall that the same propositions imply that $\tilde \cF_1\simeq  \cO_{\tilde X}( -(q-1)H+D_{V}^3) $ and
$$
\det \tilde \cF_2\simeq \cO_{\tilde X}( -(q-1)(q+2)H+q\,
D_{V}^3+D_{V}^2).
$$
Hence the required sequence follows from
$$\tilde \cF_2/\tilde \cF_1\simeq \det (\tilde \cF_2/\tilde \cF_1)\simeq (\det \tilde \cF_2)\otimes \tilde \cF_1^{-1}\simeq \cO_{\tilde X} \left(  -(q^2-1)H +(q-1)D_V^3+ D_{V}^{2}\right) . $$
\end{proof}

\medskip

\begin{lemma}\label{H^1}
$\tilde Y$ is a smooth projective
$3$-fold  with  $H^1(\cO_{\tilde Y})=0$.
\end{lemma}

\begin{proof}
Smoothness of $\tilde Y$ is a consequence of the fact that
 $\tilde \cF _2$ is a smooth $1$-foliation.

By Lemma \ref{seq-fol-dim3} we have a short exact sequence
$$0\to \cO_{\tilde X} \left(  (q^2-1)H -(q-1)D_V^3- D_{V}^{2}\right)  \to \tilde \cF _2^*\to \cO_{\tilde X} \left(  (q-1)H-D_V^3\right) \to 0.$$
But $\pi_*\cO_{\tilde X} \left(  (q^2-1)H -(q-1)D_V^3- D_{V}^{2}\right))\subset I_{Z_2}(q^2-1)$ and by Corollary \ref{Eagon-Northcott}
$H^0( I_{Z_2}(q^2-1))=0$, so $H^0( \cO_{\tilde X} \left(  (q^2-1)H -(q-1)D_V^3- D_{V}^{2}\right))=0$.
Similarly, we have $\pi_* \cO_{\tilde X} \left(  (q-1)H-D_V^3\right) \subset I_{Z_3}(q-1)$ and again by  Corollary \ref{Eagon-Northcott} we have $H^0( I_{Z_3}(q-1))=0$, so $H^0( \cO_{\tilde X} \left(  (q-1)H-D_V^3\right))=0$.
It follows that $H^0( \tilde \cF_2 ^*)=0$. Now by Lemma \ref{foliation-sequence} we have
an exact sequence
$$0\to \cO_{\tilde Y}\to \tilde f_*\cO_{\tilde X}\to \tilde  f_* \tilde \cF_2 ^*.$$
This shows that $H^0(\tilde f_*\cO_{\tilde X}/\cO_{\tilde Y})\subset H^0(\tilde f_* \tilde \cF_2 ^*)= H^0( \tilde \cF_2 ^*)=0$. But we have an exact sequence
$$H^0(\tilde f_*\cO_{\tilde X}/\cO_{\tilde Y})\to H^1(\cO_{\tilde Y})\to H^1(\tilde f_*\cO_{\tilde X})=0,$$
so $H^1(\cO_{\tilde Y})=0.$
\end{proof}

\medskip

There exists a radicial degree $p$ map $\tilde g:  \tilde Y\to \tilde X ^{(1)} $ such that
$\tilde g \circ \tilde f =F_{\tilde X/k}$.
We have $T _{\tilde X/\tilde Y}=\tilde \cF$ and  the sheaf $\cL:=T_{\tilde Y/\tilde X^{(1)}}$ is a line bundle on $\tilde Y$. We also have two short exact sequences
$$0\to \tilde \cF_2\to T_{\tilde X}\to \tilde f^*\cL\to 0$$
and
$$0\to \cL \to T_{\tilde Y}\to {(\sigma\circ \tilde g)}^*\tilde \cF_2\to 0,$$
where $\sigma:  \tilde X ^{(1)} \to \tilde X$ is the canonical map coming from base change by the Frobenius over the base field.

\begin{lemma}\label{H^0T}
We have $H^0(T_{\tilde Y})=0$.
\end{lemma}

\begin{proof}
By Lemma \ref{seq-fol-dim3} $F_{\tilde X}^*\tilde \cF_2$ has a filtration whose quotients are two line bundles, whose intesection with $H^2$ is strictly negative. Therefore $H^0(F_{\tilde X}^*\tilde \cF_2)=0$.
On the other hand, we have a short exact sequence
$$0\to \tilde \cL \to \tilde f^*T_{\tilde Y}\to F_{\tilde X}^*\tilde \cF_2\to 0,$$
where $\tilde \cL=\tilde f^*\cL .$
So it is sufficient to show that $H^0(\tilde \cL)=0$. Then $H^0( \tilde f^*T_{\tilde Y})=0$, which by the projection formula implies vanishing of $H^0( T_{\tilde Y})$.

By Proposition \ref{codim-1-foliation} we get
$$\tilde \cL\simeq \cO_{\tilde X}((q^2+q+2)H-(q+2)D_V^3-2D_V^2).$$
Let $\Pi\subset \PP (V)$ be an $\FF_q$-rational plane and let $\tilde \Pi\subset \tilde X$ be its strict transform.
The map $\tilde \Pi\to \Pi$ is the blow up of $\Pi$ at  $\Pi (\FF_q)$.
Note that  the restriction of $D^2_V$ to $\tilde \Pi$ is the strict transform $\tilde B$ of all $\FF_q$-lines contained 
in $\Pi$ and the restriction of $D^3_V$ to $\tilde \Pi$ is the exceptional divisor $E$ of $\tilde \Pi\to \Pi$.
Since $\tilde B+(q+1)E\sim (q^2+q+1)H|_{\tilde \Pi}$ we have 
$$\tilde \cL |_{\tilde \Pi}\simeq \cO_{\tilde \Pi}((q^2+q+2)H-(q+2)D_V^3-2D_V^2)\simeq \cO_{\tilde \Pi}(-q((q+1)H-E)).$$
Since $(q+1)H-E$ is nef on $\tilde \Pi$, we have  $H^0(\tilde \cL |_{\tilde \Pi})=0$.
Note that if $\Pi_1$ and $\Pi_2$ are two $\FF_q$-rational planes then they intersect along an $\FF_q$-rational line $L$.
By Lemma \ref{blow-up-sum} $\tilde X$ dominates the blow up of $\PP (V)$ along $L$, so $\tilde \Pi_1$ and $\tilde \Pi_2$ do not intersect. So we have a short exact sequence
$$0\to \tilde \cL (-D^1_V)\to \tilde\cL\to \bigoplus _{L\in \cL(V),\, \codim L=1} \tilde\cL|_{D_L}\to 0,$$
which by the above implies that
$$H^0(\tilde \cL)=H^0( \tilde \cL (-D^1_V)).$$
But $ c_1(\tilde \cL (-D^1_V))H^2<0$, so $H^0( \tilde \cL (-D^1_V))=0$.
\end{proof}

\begin{theorem}\label{new-CY'}
If $p=q=2$ then the following conditions are satisfied:
\begin{enumerate}
\item $\tilde Y$ is a smooth, projective Calabi--Yau $3$-fold,
\item $\tilde Y$ is unirational,
\item $b_2(\tilde Y)=51$ and $h^1(\cO_{\tilde Y})=h^2(\cO_{\tilde Y})=0$,
\item $h^0(T_{\tilde Y})=0$, 
\item $T_{\tilde Y}$ is not semistable with respect to some ample polarizations,
\item $\tilde Y$ does not admit a formal lifting to characteristic zero.
\end{enumerate}
\end{theorem}

\begin{proof}
Since $p=q=2$ by \cite[Corollary 3.4]{Ek} we have
$$\tilde f^*K_{\tilde Y}=K_{\tilde X}-(p-1)\, c_1(\det \tilde \cF _2)=0.$$
Therefore $pK_{\tilde Y}=0$. 

By Lemma \ref{H^1} we have $h^1(\cO_{\tilde Y})=0$. By the Riemann--Roch theorem we get
$$\chi(\cO_{\tilde Y})=\frac{1}{24}\int _{\tilde Y}c_1(T_{\tilde Y})c_2(T_{\tilde Y})=0.$$
Hence we have
$$h^0(K_{\tilde Y})=h^3(\cO_{\tilde Y})=h^2(\cO_{\tilde Y})+1\ge 1,$$
which implies that $K_{\tilde Y}=0$ and $h^2(\cO_{\tilde Y})=0$.

Let $X_0$ be the blow up of $\PP ^3$ along $\PP^3(\FF_2)$, i.e., in $15$ points.
Then $\tilde X$ is the blow up of $X_0$ along the strict transforms of all $\FF_2$-lines in $\PP ^3$, i.e.,
along $35$ disjoint smooth rational curves. Therefore $b_2(\tilde X)=1+15+35=51$.
Since $\tilde f: \tilde X\to \tilde Y$ is radicial, $\tilde Y$ is unirational and we have $b_2(\tilde Y)=b_2(\tilde X)=51$.

By Lemma \ref{H^0T} we have $h^0(T_{\tilde Y})=0$.

Let us set $g=\sigma\circ \tilde g$. Since $\tilde f \circ   g =F_{\tilde Y}$ we have
$$p\mu_{ g^*H} (\cL)= \mu_{ g^ *H}(F^*_{\tilde Y} \cL)=\deg \tilde g\cdot \mu_H(\tilde f^*\cL)=
p\cdot  c_1(\tilde f^*\cL)H^2=p(q^2+q+2 )>0.$$
In particular, $\mu_{g^*H} (\cL)>\mu_{ g^*H} (T_{\tilde Y})=0$ and $T_{\tilde Y}$ is ${g^*H}$-unstable.
Then $T_{\tilde Y}$ is $A$-unstable for ample polarizations $A$ that are close to the divisor $\tilde g^*H$ (which is only nef and big) in the (rational) Neron-Severi group of $\tilde Y$.

The fact that $\tilde Y$ does not admit a formal lifting to characteristic zero follows by the same arguments as that of 
\cite[Theorem 2.1]{Sc}.
\end{proof}

\begin{remark}
One can expect that similarly to Hirokado's example, the above example is arithmetically rigid (see \cite[Theorem A]{Ek}). Analogously to [ibid.], one should compute all the Hodge numbers of $\tilde Y$ and prove that one can lift a group of automorphisms, but it seems that there is no short and easy way to do so. 
\end{remark}

\section{Cones over projective spaces}\label{cone}

Let us recall the following lemma which comes from \cite[Lemma
4.14]{Wa2} (the proof works also in positive characteristic and in higher dimensions).

\begin{lemma} \label{Wahl's reflexivity} Let $Z$ be an affine variety
  with singularity at $z$ and let $f:\tilde Z\to Z$ be resolution of
  singularities.  Let us assume that $f$ is an isomorphism on
  $Z-\{z\}$. Let $M$ be a line bundle on $\tilde Z$ such that
  $L^{\otimes m}\simeq \cO_{\tilde Z}(D)$ for some positive integer
  $m$ and an effective exceptional divisor $D$. Then $f_*L$ is
  reflexive.
\end{lemma}

\medskip
Let $g: \tilde Z=\VV (\cO_{\PP ^m}(-d))=\Spec \bigoplus _{i\ge 0} \cO_{\PP ^m}(d)^{\otimes i}\to \PP^m$ be
the total space of $\cO_{\PP ^m}(d)$ and let $S$ be its zero
section.  Let $f:\tilde Z \to Z=\Spec \bigoplus _{i\ge 0} H^0(\PP^m,
\cO_{\PP ^m}(d)^{\otimes i})$ be the contraction of $S$ to a cone over $(\PP^m, \cO_{\PP ^m}(d))$.

\begin{proposition} \label{main-2'}
  $Z$ has only klt singularities and $f_*\Omega_{\tilde Z} (\log S)$
  is reflexive.  If the base field has positive characteristic $p$ and
  $p$ divides $d$ then $f_*\Omega_{\tilde Z}$ is not reflexive. Moreover, $f_*\Omega_{\tilde Z} (\log S)$
is locally free if $p=d=2$ and $m=1$.
\end{proposition}

\begin{proof}
 A simple computation shows that $K_{\tilde Z}+S=f^*K_Z+\frac{m+1}{d}S$, so $Z$ has only klt singularities.
    Let us consider an exact sequence
$$0\to \Omega_{\PP^m}^1\to G\to \cO_{\PP^m}\to 0$$
determined by the class of $\cO_{\PP^m}(d)$ in $H^1(\Omega_{\PP^m}^1)$ via
$H^1(\cO^*_{\PP^m})\to H^1(\Omega_{\PP^m}^1)$.  By \cite[Proposition
3.3]{Wa} we have $\Omega_{\tilde Z} (\log S)\simeq g^*G$.  If $p$
divides $d$ then the above sequence splits and $\Omega_{\tilde Z}
(\log S)\simeq  g^*\Omega_{\PP^m}\oplus \cO_{\tilde Z}$.  If $p$ does
not divide $d$ then the above sequence does not split and
$\Omega_{\tilde Z} (\log S)\simeq g^*\cO_{\PP^m}(-1)^{\oplus (m+1)}$.

Since $ g^*\cO_{\PP^m}(-d)\simeq \cO_{\tilde Z}(S)$,
Lemma \ref{Wahl's reflexivity} shows that $M=f_*(g^*\cO_{\PP^m}(-1))$
is a  rank $1$ reflexive sheaf (corresponding to the ideal sheaf of a hyperplane passing
through the vertex $z$ of the cone $Z$).
Euler's exact sequence
$$0\to \Omega_{\PP^m}^1\to \cO_{\PP ^m}(-1)^{\oplus (m+1)}\to \cO_{\PP^m}\to 0$$
leads to the exact sequence
$$0\to f_*(g^*\Omega_{\PP^m}^1)\to M^{\oplus (m+1)}\to \cO_{Z},$$
which shows that $ f_*(g^*\Omega_{\PP^m}^1)$ is reflexive.
Therefore $f_*\Omega_{\tilde Z} (\log S)$ is always reflexive.
More precisely, $f_*\Omega_{\tilde Z} (\log S)\simeq f_*(g^*\Omega_{\PP^m}^1)\oplus \cO_{Z} $
if $p$ divides $d$ and $f_*\Omega_{\tilde Z} (\log S)\simeq M^{\oplus (m+1)}$ if $p$
does not divide $d$. In particular, $f_*\Omega_{\tilde Z} (\log S)$ is
locally free if  $p=d=2$ and $m=1$.

Now let us note that the map $ g^*\Omega_{\PP^m}^1\to \Omega_{\tilde
  Z} (\log S)$ factors through the canonical map
$g^*\Omega_{\PP^m}^1\to \Omega_{\tilde Z}$. Let us assume that $p$
divides $d$.  Then $ g^*\Omega_{\PP^m}^1\to \Omega_{\tilde Z} (\log
S)$ splits and hence the sequence
$$0\to g^*\Omega_{\PP^m}^1\to \Omega_{\tilde Z}\to \Omega_{\tilde Z/\PP^m}\to 0$$
also splits. It follows that  $ \Omega_{\tilde Z}\simeq \cO_{\tilde Z}(-S) \oplus g^*\Omega_{\PP^m}^1$,
so  $f_* \Omega_{\tilde Z}\simeq  m_z\oplus f_*(g^*\Omega_{\PP^m}^1)$ is not reflexive.
\end{proof}

The above proposition immediately implies Proposition \ref{main-2}.

\section{Purely inseparable flops}

In this section we construct an analogue of the celebrated Atiyah's flop. Let us recall that Atiyah's flop is obtained from
a quadric singularity $(xy=zw)\subset \AA^4$ by blowing up along the planes $x=z=0$ and $x=w=0$.
The corresponding rational map $X_1\dashrightarrow X_2$ between the blow-ups is resolved by the blow up of the singularity
along the point $x=y=z=w=0$. Geometrically, this rational map is obtained by taking $\PP ^1\subset X_1$ with normal bundle
$\cO_{\PP ^1}(-1)^{\oplus  2}$, blowing it up and then contracting the exceptional divisor $\PP ^1\times \PP ^1$ in the other
direction to $\PP^1$ with the same normal bundle  $\cO_{\PP ^1}(-1)^{\oplus  2}$. In fact, this flop can be seen also when studying
the standard Cremona transformation $\PP ^3\dashrightarrow \PP^3$ defined by $[x_0,x_1,x_2,x_3]\to [x_0^{-1},x_1^{-1},x_2^{-1},x_3^{-1}]$
(see, e.g., \cite[p.~6014]{Le}). Similarly, our ``purely inseparable Cremona transformation'' leads to a purely inseparable flop
described below.

\medskip

Let $X_1$ be a smooth $3$-fold defined over an algebraically closed field of characteristic $p$ and let us fix some $q=p^e$.
Let $C_1\subset X_1$ be a smooth rational curve with normal bundle $N_{C_1|X_1}\simeq \cO_{\PP ^1}(-q)^{\oplus  2}$ and let
$p_1:X\to X_1$ be the blow up of $X_1$ along $C_1$. The exceptional divisor $E_1\simeq \PP ^1\times \PP ^1$ has normal bundle
$N_{E_1|X}\simeq \cO_{\PP ^1\times \PP^1}(-1,-q)$. Assume that there exists a purely inseparable morphism  $\varphi: X\to X$ exchanging rulings, so
that $E_2=\varphi(E_1)$ has normal bundle $N_{E_2|X}\simeq \cO_{\PP ^1\times \PP^1}(-q,-1)$. Then contracting $E_2$ along the other ruling
we get $p_2:X\to X_2$.

Note that since the conormal bundle of $E_1$ in $X$ is ample, we can contract $E_1$ to a point, i.e., there exists a morphism
$f_1: X\to Y$ to a normal algebraic space $Y$ that is an isomorphism outside of $E_1$ and contracts $E_1$ to a point
(see \cite[Corollary 6.12]{Ar}).
Clearly, $f_1$ factors through $p_1$ inducing a contraction $q_1:X_1\to Y$. Similarly, we can construct $q_2: X_2\to Y$. Then
$\varphi: X\to X$ induces a purely inseparable endomorphism $\varphi_Y:Y\to Y$. Summing up, we have the following diagram:
$$\xymatrix{
X \ar@/_1pc/[dd]_{f_1}\ar[r]^{\varphi}\ar[d]^{p_1} & X\ar[d]_{p_2}\ar@/^1pc/[dd]^{f_2}\\
X_1 \ar@{-->}[r]\ar[d]^{q_1} & X_2\ar[d]_{q_2}\\
Y \ar[r]^{\varphi_Y}& Y\\
}$$
in which the rational map $X_1\dashrightarrow X_2$ is a purely inseparable map  with indeterminacy locus equal to $C_1$.
Geometrically, this rational map replaces $C_1$ with another smooth rational curve $C_2$ with normal bundle
$N_{C_2|X_2}\simeq \cO_{\PP ^1}(-q)^{\oplus  2}$.

We could also directly try to contract $E_1\subset X$ along the other ruling obtaining maps $f_1:X\to Z_1$ and $g_1:Z_1\to Y$, but then $Z_1$ acquires rather bad singularities along the curve $f_1(E_1)$. $Z_1$ comes with a purely inseparable morphism $Z_1\to X_2$ compatible with
$\varphi$ and $\varphi_Y$.

\medskip

In fact, by \cite[Corollary 6.13]{Ar} the above $E_1\subset X$ is locally in the \'etale topology isomorphic
to the cylinder over $(\PP ^1\times \PP^1, \cO_{\PP ^1\times \PP^1}(1,q))$. Then locally $Y$ is simply the cone over
$(\PP ^1\times \PP^1, \cO_{\PP ^1\times \PP^1}(1,q))$. Then the map $\varphi _Y: Y\to Y$ is induced from the map
$\PP ^1\times \PP^1\to \PP ^1\times \PP^1$ sending $(x,y)$ to $(y^q,x)$.

\section*{Appendix: Plane curves in positive characteristic}

Let us fix $\bar k$-points $\Sigma=\{P_1,...,P_{m}\}$ of $\PP^2$ and
some positive integers $(q_1,...,q_m)$.  Let $\pi:X\to \PP^2$ be the
blow up of $\PP^2$ along $\Sigma$. Let us fix some positive integer
$d$ and let us set $$B=(d+1)H-\sum q_iE_i,$$ where $H$ is the pull
back of a line on $\PP ^2$ and $E=\sum E_i$ is the exceptional divisor
of $\pi$.  The projection formula implies that
$$\chi (X, K_X+B)=\chi(\PP^2 , m_{P_1}^{q_1-1}...m_{P_m}^{q_m-1}(d-2)) =\binom{d}{2}-\sum_{i=1}^m\binom{q_i}{2}.$$
Moreover, we have
$$H^i (X, K_X+B)=H^i(\PP^2 , m_{P_1}^{q_1-1}...m_{P_m}^{q_m-1}(d-2))$$
for $i=0,1,2.$

\medskip

Let $C\subset \PP^2$ be a reduced plane curve of degree $d$. Let us take as
$\Sigma=\{P_1,...,P_{m}\}$  the singular locus of $C$ and let $q_i$
be the multiplicity of $P_i$ in $C$.
Below we recall a certain lemma contained in a part of \cite[Exercise
6.8]{Laz}.

\begin{lemma} \label{Lazarsfeld}
There exists an ample $\QQ$-divisor $M$ on $X$ such that
$$H^1(X, K_X+\lceil M\rceil)=H^1(\PP^2 , m_{P_1}^{q_1-1}...m_{P_m}^{q_m-1}(d-2)).$$
\end{lemma}

\begin{proof}
  Let $\tilde C$ be the strict transform of $C$. Then $\pi^*C=\tilde C+\sum
  q_iE_i$.  Let us choose a small $\epsilon>0$ such that $H-\epsilon
  E$ is ample. For some $a$ let us consider the $\QQ$-divisor
$$M=(d+1)H- a\tilde C - \sum a(q_i+\epsilon)E_i .$$
If $0<a<1$ then the numerical equivalence
$$M\equiv (d+1)(1-a)H +a(H-\epsilon E)+a(dH-\tilde C -\sum q_iE_i) \equiv (d+1)(1-a)H +a(H-\epsilon E)$$
shows that $M$ is ample. Now let us choose $a<1$ such that
$a>\frac{q_i}{q_i+\epsilon}$ for every $i$.  Then
$$B=\lceil M\rceil =(d+1)H-\sum q_i E_i,$$
so $$H^1(X, K_X+\lceil M\rceil)= H^1(\PP^2 , m_{P_1}^{q_1-1}...m_{P_m}^{q_m-1}(d-2))$$
as required.
\end{proof}

\medskip

\begin{remark}\label{ncd}
Note that even though the support of $\Delta=\lceil M\rceil -M$ need
not be a simple normal crossing divisor, the proof of Sakai's lemma
\cite[Exercise 6.6]{Laz} shows that if $\tau : X'\to X$ is a log
resolution of $(X, \Delta )$ then vanishing of $H^1(X', K_{X'}+\lceil
\tau^* M\rceil )$ implies vanishing of $H^1(X, K_{X}+\lceil M\rceil
)$. Let $E'$ be the exceptional divisor of $\tau$.  We can choose a
small $\epsilon '>0$ such that $M'=\tau^* M-\epsilon 'E'$ is ample and
$\lceil M'\rceil=\lceil \tau^* M\rceil$. Therefore if $H^1(X,
K_{X}+\lceil M\rceil )\ne 0$ then there exists an ample divisor $M'$
on $X'$ such that its fractional part $\Delta'=\lceil M'\rceil -M'$ is
a simple normal crossing divisor and $H^1(X', K_{X'}+\lceil M'\rceil
)\ne 0$.
\end{remark}

\medskip

\begin{question}
In the notation as above:
\begin{enumerate}
\item Is it true that $\Sigma$ imposes
independent conditions on curves of degrees $s \ge d-2$?
\item More generally, is it true that
$$H^1(\PP^2 , m_{P_1}^{q_1-1}...m_{P_m}^{q_m-1}(s))=0 $$
for $s\ge d-2$?
\item If $\binom{d}{2}\le \sum_{i=1}^m\binom{q_i}{2}$ (corresponding to $\chi (Y, K_Y+B)\le 0$)
then can the linear system $ |m_{P_1}^{q_1}...m_{P_m}^{q_m}(d)|$ contain any reduced curves?
\end{enumerate}
\end{question}

Clearly, in characteristic zero by the Kawamata--Viehweg vanishing  the answers to  questions 1 and 2 are positive
and  the answer to question 3 is negative (see \cite[Exercise 6.8]{Laz}). By \cite[Theorem 1.1]{CT} Kawamata--Viehweg
vanishing fails on rational surfaces in positive characteristic. However, one can still hope that answers
to the above questions are the same as in characteristic zero.

\subsection*{Acknowledgement}

The author would like to thank. H. Esnault, R. Pink, K. Schwede, V. Srinivas,
A. Werner and D. Zhang for useful conversations.  In particular, the
idea of proof of the first part of Theorem \ref{main-1} is entirely
due to H. Esnault and V. Srinivas, although the author had not seen
their proof of ampleness of $K_{\bar X}+D$.

\end{document}